\documentclass{amsart}

\usepackage{amsmath}
\usepackage{amsthm,amssymb,color,comment}
\usepackage{bbm}
\usepackage{hyperref}
\usepackage[TS1,T1]{fontenc}
\usepackage{inputenc}
\usepackage{dsfont}
\usepackage{tikz}
\usepackage{enumitem}
\usepackage{soul}
\usepackage[sort]{cite}

\numberwithin{equation}{section}

\newtheorem{thm}{Theorem}[section]

\newtheorem{lem}[thm]{Lemma}
\newtheorem{cor}[thm]{Corollary}

\newtheorem{Def}[thm]{Definition}

\theoremstyle{definition}
\newtheorem{Ass}[thm]{Assumption}
\newtheorem{rem}[thm]{Remark}

\newtheorem{Not}[thm]{Notation}

\DeclareMathOperator{\DIV}{div}

\newcommand{\R}{\mathbb{R}}
\newcommand{\N}{\mathbb{N}}
\newcommand{\p}{\partial}
\newcommand{\eps}{\varepsilon}

\newcommand{\supp}{\text{supp}}
\newcommand{\diff}{\mathop{}\!\mathrm{d}}

\newcommand{\symm}{\mathbb{R}^{d\times d}_{\text{sym}}}

\makeatletter
\newcommand{\doublewidetilde}[1]{{%
  \mathpalette\double@widetilde{#1}%
}}
\newcommand{\double@widetilde}[2]{%
  \sbox\z@{$\m@th#1\widetilde{#2}$}%
  \ht\z@=.9\ht\z@
  \widetilde{\box\z@}%
}
\makeatother
\author{Miroslav Bul\'i\v{c}ek}
\address{Mathematical Institute, Faculty of Mathematics and Physics, Charles University, Sokolovsk\'{a} 83, 186 75, Prague, Czech Republic}
\email{mbul8060@karlin.mff.cuni.cz}
\thanks{Miroslav Bul{\'\i}{\v{c}}ek was supported by the project No. 20-11027X financed by GA\v{C}R.}

\author{Piotr Gwiazda}
\address{Institute of Mathematics of Polish Academy of Sciences, Jana i J\k edrzeja \'Sniadeckich 8, 00-656 Warsaw, Poland}
\email{pgwiazda@mimuw.edu.pl}
\thanks{Piotr Gwiazda was supported by National Science Center, Poland through project no. 2018/31/B/ST1/02289.}

\author{Jakub Skrzeczkowski}
\address{Faculty of Mathematics, Informatics and Mechanics, University of Warsaw, Stefana Banacha 2, 02-097 Warsaw, Poland}
\email{jakub.skrzeczkowski@student.uw.edu.pl}
\thanks{Jakub Skrzeczkowski was supported by National Science Center, Poland through project no. 2019/35/N/ST1/03459.}

\author{Jakub Woźnicki}
\address{Faculty of Mathematics, Informatics and Mechanics, University of Warsaw, Stefana Banacha 2, 02-097 Warsaw, Poland}
\email{jw.woznicki@student.uw.edu.pl}
\thanks{Jakub Woźnicki was supported by National Science Center, Poland through project no. 2017/27/B/ST1/01569}

\begin{document}

\title[Non-Newtonian fluids with discontinuous-in-time stress tensor]{Non-Newtonian fluids with discontinuous-in-time stress tensor}

\begin{abstract}
We consider the system of equations describing the flow of incompressible fluids in bounded domain. In the considered setting, the Cauchy stress tensor is a monotone mapping and has asymptotically $(s-1)$-growth with the parameter $s$ depending on the spatial and time variable. We do not assume any smoothness of $s$ with respect to time variable and assume the log-H\"{o}lder continuity with respect to spatial variable. Such a setting is a natural choice if the material properties are instantaneously, e.g. by the switched electric field. We establish the long time and the large data existence of weak solution provided that $s\ge(3d+2)(d+2)$.
\end{abstract}


\maketitle

\section{Introduction}
\noindent We consider the system of partial differential equations
\begin{align}\label{sys:main_system}
    \left\{\begin{array}{ll}\partial_t u(t, x) + \DIV_x (u(t,x)\otimes u(t, x)) + \nabla_x p(t, x) = \DIV_x S(t, x, Du) + f(t, x)\\
    \DIV_x u(t, x) = 0
    \end{array}\right.
\end{align}
describing the flow of incompressible, homogeneous, non-Newtonian fluid. Here, $u$ is the velocity of the fluid, $p$ denotes pressure, $S$ is the constitutively determined part of the Cauchy stress tensor depending on the symmetric gradient $Du$, $f$ represents a given density of the  external body forces and for simplicity, we consider that the density of the fluid is equal to one. The system of equations \eqref{sys:main_system} is formulated on a space-time cylinder $\Omega_T:=(0,T)\times \Omega$, where $\Omega \subset \mathbb{R}^d$ is a Lipschitz domain and $d$ denotes the dimension. The above system is completed by the no-slip boundary conditions, i.e., $u$ vanishes on $\partial \Omega$, and by the initial condition $u_0(x)$.\\

\noindent Most of the papers devoted to the analysis of non-Newtonian fluids assume Cauchy stress tensor to be of power type
$$
S(t,x,Du) \sim (\nu_0 + \nu_1\,|Du|^{p-2}) \, Du.
$$
First existence results concerning \eqref{sys:main_system} were proved for $p \geq \frac{11}{5}$ (in 3D) by Lions and Ladyzhenskaya in \cite{MR0259693, MR0254401}. Since then, many improvements have appeared, from the higher regularity method in \cite{malek1996weakandmeasure} giving the bound $p\ge \frac{9}{5}$, followed by the $L^{\infty}$-truncation method for $p\ge \frac{8}{5}$, see \cite{MR1713880}, the Lipschitz truncation method for $p>\frac{6}{5}$, see \cite{MR2668872,MR2001659,MR3023393}, up to a new definition of a solution in \cite{MR4102807} leading to the theory for all $p>1$.  We refer to the extensive review in \cite{MR4076814} in context of fluids with very complicated rheology. Nevertheless, let us remark that such equations are still a topic of research - recently they have been analyzed in the context of nonuniqueness and convex integration \cite{MR4328053}, extending the groundbreaking paper of Buckmaster and Vicol on Navier-Stokes equation \cite{MR3898708}. 

\noindent In this paper, we are interested in the case when $S$ has the so-called non-standard growth. The iconic example is
\begin{equation}\label{eq:stress_tensor_var_exp}
S(t,x,Du) \sim (\nu_0 + \nu_1\,|Du|^{s(t,x)-2}) \, Du,
\end{equation}
where the exponent $s(t,x)$ depends on the time variable $t$ and the spatial variable $x$. The motivation for considering~ \eqref{eq:stress_tensor_var_exp} comes from the behaviour of electrorheological fluids whose mechanical properties dramatically change when an external electric field is applied, see~\cite{RAJAGOPAL1996401}. The topic has been extensively studied over the last years from the mathematical point of view, here we refer to~\cite{diening2011lebesgue,MR1930392, MR1810360, MR2037246, MR3474486, MR2346460, MR2610563}. These considerations have been recently generalized to the micropolar fluids~\cite{MR3474486, MR3373570} and also chemically reacting fluids~\cite{MR4286250, MR3904139, MR3834434, ko2018existence}.\\

\noindent A natural assumption on $S$, which reflects the structure \eqref{eq:stress_tensor_var_exp} involves the growth and the coercivity formulated as the following inequality
\begin{equation}\label{eq:ass_in_intro}
c \, S(t, x, \xi) : \xi \geq | \xi|^{s(t, x)} + | S(t, x, \xi)|^{s'(t, x)} - h(t, x),
\end{equation}
where $c$ is some constant, $h \in L^{1}(\Omega_T)$ and $s'(t,x)$ is the H{\"o}lder conjugate exponent  to $s(t,x)$, i.e., $s':=\frac{s}{s-1}$. Analysis of \eqref{sys:main_system} with the Cauchy stress tensor of the form \eqref{eq:ass_in_intro} requires the concept of generalized Lebesgue spaces $L^{s(t,x)}$. One can also generalize \eqref{eq:ass_in_intro} by replacing power-type function with a generalized $N$-function. The resulting analysis requires application of the general Musielak--Orlicz spaces \cite{chlebicka2019book} as in \cite{MR3007700,MR3190317,MR2466805, MR2597212}. We remark that all results of this paper can be formulated in this setting but we decided not to do so for the sake of clarity. \\

\noindent In this paper we establish the existence of global-in-time and large-data solutions to \eqref{sys:main_system} for exponents $s(t,x)$ being discontinuous in the time variable. The former approaches were based on the so-called log-H{\"o}lder continuity of the exponent $s(t,x)$, which allows one to use the density of smooth functions in the space $L^{s(t,x)}$,  see \cite{MR3847479, cruz2013variable}. In fact, the log-H{\"o}lder continuity is necessary for the density to be true, see \cite[Example 6.12]{cruz2013variable}. Nevertheless, inspired by \cite{bulicek2021parabolic}, where the following problem was treated
$$
\partial_t u + \DIV \left(|\nabla u|^{s(t,x)-2} \nabla u \right) = f
$$
with $s(t,x)$ being discontinuous in time and log-H{\"o}lder continuous in space (extending theresults of \cite{chlebicka2019parabolic}),
we do not require the smoothness of $s$ with respect to the time variable here. In addition, we also do not require any relationship between the minimal and maximal values of $s(t,x)$. The only restriction is due to the convective term and has the form $s\ge \frac{3d+2}{d+2}$. Note that in case we consider a generalized Stokes problem only, i.e., we consider the system \eqref{sys:main_system} without the term $\DIV_x (u(t,x)\otimes u(t, x))$, there is no restriction on $s$ except the log-H\"{o}lder continuity with respect to the spatial variable.  \\

\noindent When compared to \cite{bulicek2021parabolic} the main difficulty of the present work lies in the fact that \eqref{sys:main_system} can be tested only with a divergence-free function. In particular, we cannot test it with the truncation of solution as it loses divergence-free property after applying truncation operator. Even when one recovers pressure $p$ by the Ne\v{c}as Theorem, one obtains terms which are not treatable as we have only weak convergence of both the pressure and derivatives of the solution. We remark that one can try to overcome this problem by applying approximation called Lipschitz truncation method, see \cite{MR3672391,MR3119635,MR2943947,MR2394508,MR970512}. Nevertheless, this approach does not seem to be applicable here as our work uses equation satisfied by the exploited approximation in the crucial way. Contrary to the mollification, as Lipschitz truncations are defined by the maximal function, it is not trivial to write equation satisfied by them.\\

\noindent The structure of the paper is as follows. In Section~\ref{section:2}, we introduce the precise assumption on the Cauchy stress tensor $S$ and also on the variable exponent $s(t,x)$ and formulate rigorously the main result of the paper. Section~\ref{section:3} is devoted to the approximation theory in Bochner-Sobolev spaces related to our problem. Then we deal with an abstract parabolic equation in Section~\ref{section:4} and show the validity of the local energy inequality based on the proper pressure decomposition and proper mollification procedure. Then in Section~\ref{section:5} we introduce a classical approximative problem for which the theory is available and in Section~\ref{section:6} we finally pass to the limit and finish the proof of the main theorem. The classical but auxiliary tools are recalled for the sake of clarity in Appendix.

\section{Preliminaries and the main result}\label{section:2}
\noindent First, we introduce the notation used through the work. Let $d = 2, 3$ be the dimension of the space, $\Omega \subset \mathbb{R}^d$ be a Lipschitz domain and $T > 0$ denotes the length of time interest. We write $x$ for an element of $\Omega$ and $t$ for an element of $(0,T)$.   The corresponding parabolic domain will be denoted with $\Omega_T:= (0,T) \times \Omega$ and the its specific subdomains as $\Omega_t = (0,t) \times \Omega$. Next, for any $a,b\in \mathbb{R}^d$ we write $a\cdot b$ for the standard scalar product of $a$ and $b$. Similarly, the space $\symm$ denotes the space of symmetric $d\times d$ matrices and for any $A,B \in \symm$ we denote the scalar product by $A:B$. In addition, the symbol $\otimes$ is reserved for the tensorial product, i.e., for $a,b\in \mathbb{R}^d$ we denote $a\otimes b \in \symm$ as $(a\otimes b)_{ij}:=a_ib_j$ for $i,j=1,\ldots, d$. We use the standard notation for Sobolev and Lebesgue function space and frequently do not distinguish between scalar-, vector- or matrix-valued functions.
In addition, to shorten the notation, we frequently use the following simplifications. When $f \in L^p(\Omega)$, we simplify it to $f \in L^p_x$. Similarly, if $f \in L^p(0,T; L^q(\Omega))$, $f \in L^p(0,T; W^{1,q}(\Omega))$ or $f \in L^p(0,T; W^{1,q}_0(\Omega))$, then we write $f \in L^p_t L^q_x$, $f \in L^p_t W^{1,q}_x$ or $f \in L^p_t W^{1,q}_{0,x}$ respectively (here, $W^{1,q}(\Omega)$ and $W^{1, q}_0(\Omega)$ are the usual Sobolev spaces). For the exponent $p \in [1,\infty]$, we denote by $p'$ its H{\"o}lder conjugate defined by the equation $\frac{1}{p} + \frac{1}{p'} = 1$. Symbol $\nabla u$ is used for the spatial gradient of $u$ and $Du$ denotes its symmetric part, i.e. $Du = \left(\nabla u + (\nabla u)^\intercal\right)/2$. We mostly work in the variable exponent spaces $L^{s(t,x)}(\Omega_T)$. The last special function space related to the fluid mechanics is $ L^2_{0,\DIV}(\Omega)$, which is defined as a closure of the set $\{u\in C_c(\Omega: \mathbb{R}^d) \, \DIV u=0\}$ in $L^2(\Omega)$. Through the paper we also employ the universal constant $C$ that may vary from line to line, but depends only on data. In some particular cases we write such dependence explicitly. For the definition and the discussion on properties of the aforementioned spaces see Appendix~\ref{app:musielaki}.

\subsection*{Assumptions on data}
Let us now state the needed assumptions on the exponent function $s(t, x)$:
\begin{Ass}\label{ass:exponent_cont_space}
We assume that a measurable function $s(t,x): \Omega_T \to [1,\infty)$ satisfies the following:
\begin{enumerate}[label=(A\arabic*)]    \item\label{ass:cont} (continuity in space) $s(t,x)$ is a log-H\"older continuous functions on $\Omega$ uniformly in time, i.e. there is a constant $C$ such that for all $x, y \in \Omega$ and all $t \in [0,T]$
$$
|s(t,x) - s(t,y)| \leq -\frac{C}{\log|x-y|},
$$
\item\label{ass:usual_bounds_exp} (bounds) it holds that $ \frac{3d + 2}{d+2} =: s_{\text{min}}\leq s(t,x)  \leq s_{\text{max}}$ for a.e. $(t,x) \in \Omega_T$.
\end{enumerate}
\end{Ass}
For later purposes we also define an exponent $s_0$ as
\begin{equation} \label{s0}
s_0 := 3 + \frac{2}{d}.
\end{equation}
We remark that the condition \ref{ass:cont} is somehow standard, as it guarantees good approximation properties (with respect to spatial variable $x$) in the variable exponent space. Assumption \ref{ass:usual_bounds_exp} is related to the  continuity of the three-linear form
$$
v\longmapsto \int_{\Omega_T} v\otimes v : \nabla v\diff x\diff t
$$
that appears in the analysis and is somehow ``necessary" to obtain the so called energy equality. Note that the same problem appears in the classical Navier--Stokes equations. \\

\subsection*{Assumptions on the stress tensor}
Concerning the stress tensor $S: (0, T) \times \Omega\times \symm \to \symm$ we assume the following:
\begin{Ass}\label{ass:stress_tensor} We assume that
\begin{enumerate}[label=(T\arabic*)]
\item \label{T1}$S(t,x, \xi)$ is a Carath\'{e}odory function and $S(t, x, 0) = 0$,
\item \label{coercitivity_stress_tensor} (coercivity and growth conditions) There exists a positive constant $c$  and a non-negative, integrable function $h(t, x)$, such that for any $\xi\in\symm$ and almost every $(t, x) \in \Omega_T$
$$
c \, S(t, x, \xi) : \xi \geq | \xi|^{s(t, x)} + | S(t, x, \xi)|^{s'(t, x)} - h(t, x)
$$
\item \label{monotonicity_stress_tensor}(monotonicity) S is monotone, i.e.:
$$
(S(t, x, \xi_1) - S(t, x, \xi_2)) : (\xi_1 - \xi_2) \geq 0
$$
for all $\xi_1 \neq \xi_2\in \symm$ and almost every $(t, x) \in \Omega_T$.
\end{enumerate}
\end{Ass}
Notice that the assumption \ref{T1} is assumed just to simplify the approximating scheme. On the other hand, the monotonicity \ref{monotonicity_stress_tensor} is the key property. One could consider a more general setting of maximal monotone graphs here without any difficulties, but to avoid the technical difficulties we consider $S$ to be a Carath\'{e}odory mapping. Finally, the assumption~\ref{coercitivity_stress_tensor} is a natural setting to get bounds for gradient of an unknown velocity as well as bound on the stress tensor.

\subsection*{Main result}
Having introduced all the needed notation, we may finally state the main theorem.
\begin{thm}\label{thm:the_main_result}
Let $S(t, x, \xi)$ satisfy the Assumption~\ref{ass:stress_tensor} with the exponent $s(t, x)$ satisfying Assumption~\ref{ass:exponent_cont_space}. Then, for any $f\in L^{1}_t L^2_x \cup L^{s_{\text{min}}'}(0,T;(W_0^{1,s_{\text{min}}}(\Omega))^*)$ and any $u_0\in L^2_{0,\DIV}(\Omega)$, there exists $u\in L^\infty_t L^2_x \cap L^{s_{\text{min}}}_t W^{1, s_{\text{min}}}_{0, x}$, $Du\in L^{s(t, x)}(\Omega_T)$, such that $\DIV u=0$ almost everywhere in $\Omega_T$ and
\begin{align}\label{eq:the_main_result}
    \int_{\Omega_T}-u  \cdot \partial_t \phi - u\otimes u : \nabla\phi + S(t, x, Du):D\phi\diff x\diff t = \int_{\Omega_T}f \cdot \phi\diff x\diff t + \int_{\Omega}u_0(x)\cdot \phi(0,x)\diff x
\end{align}
for any $\phi\in C^\infty_c([0, T) \times \Omega)$ fulfilling $\DIV \phi = 0$ almost everywhere in $\Omega_T$.
\end{thm}
We would like to point out here that the assumption  $f\in L^{1}_t L^2_x \cup L^{s_{\text{min}}'}(0,T;(W_0^{1,s_{\text{min}}}(\Omega))^*)$ can be relaxed, namely the second part. But because we do not want to complicate the paper we do not consider it here. Even more, in the proof we consider only the case $f\in L^{1}_t L^2_x$ to avoid difficulties. Finally, we want to emphasize again that the assumption~\ref{ass:usual_bounds_exp} is not needed in case of generalized Stokes system, i.e., if the term  $\int_{\Omega_T}u\otimes u : \nabla\phi$ is omitted.

\section{Approximation in variable exponent spaces}\label{section:3}
\noindent In this section we discuss a method to approximate functions $u$ such that $D u \in L^{s(t,x)}(\Omega_T)$ in spatial variable. Moreover, we will guarantee that the convergence $Du^{\varepsilon} \to  Du$ holds in $L^{s(t,x)}(\Omega_T)$.

\begin{Def}[Mollification with respect to the spatial variable]\label{res:mol_in_sp}
Let $\eta:\R^d \to \R$ be a standard regularizing kernel, i.e. $\eta$ is a smooth, nonnegative function, compactly supported in a ball of radius one and fulfills $\int_{\mathbb{R}^d} \eta(x) \diff x = 1$. Then, we set $\eta_{\varepsilon}(x) = \frac{1}{\varepsilon^d} \eta\left(\frac{x}{\varepsilon}\right)$ and for arbitrary $u: \R^d \times [0,T] \to \mathbb{R}$, we define
$$
u^{\varepsilon}(t,x) = \int_{\R^d} \eta_{\varepsilon}(x-y) u(t,y) \diff y.
$$
\end{Def}

\noindent For further reference we also introduce mollification with respect to the time variable.

\begin{Def}[Mollification with respect to time]\label{res:mol_in_ti}
Let $\zeta:\R \to \R$ be a standard regularizing kernel, i.e. $\zeta$ is a smooth, nonnegative function, compactly supported in a ball of radius one and fulfills $\int_{\mathbb{R}} \zeta(x) \diff x = 1$. Then, we set $\zeta_{\varepsilon}(x) = \frac{1}{\varepsilon} \zeta\left(\frac{x}{\varepsilon}\right)$ and for arbitrary $u: \R \times \Omega \to \R$, we define $\mathcal{R}^{\varepsilon}u: \R \times \Omega \to \R$ as
$$
\mathcal{R}^{\varepsilon}u(t,x) = \int_{\R} \zeta_{\varepsilon}(t-s)\, u(s,x) \diff s.
$$
\end{Def}

\noindent The main result of this section reads:
\begin{thm}\label{thm:approx_theorem}
Let $\Omega \subset \R^d$, and $\psi:\Omega \to \R$ be arbitrary such that $\psi \in C_c^{\infty}(\Omega)$. Let $u$ satisfy $u \in L^{\infty}_t L^2_x \cap L^1_t W^{1,1}_{0,x}$ and $D u \in L^{s(t,x)}(\Omega_T)$. If we extend $u(t,x)$ by zero for all $x\notin \Omega$, then there exists $\varepsilon_0>0$ depending on $\psi$ and $\Omega$ such that:
\begin{enumerate}[label=(S\arabic*)]
\item\label{thmitem:conv1} $\left(u^{\varepsilon} \psi\right)^{\varepsilon} \in L^{\infty}(0,T;C_0^{\infty}(\Omega))$ for all $\varepsilon \in (0, \varepsilon_0)$,
\item\label{thmitem:conv2} $\left(u^{\varepsilon} \psi\right)^{\varepsilon}  \to u\, \psi$ a.e. in $\Omega_T$ and in $L^1(0,T; L^1(\Omega))$ as $\varepsilon \to 0^+$,
\item\label{thmitem:conv3}
$D\left(u^{\varepsilon} \psi\right)^{\varepsilon} \in L^{s(t,x)}(\Omega_T)$ and
$D\left(u^{\varepsilon} \psi\right)^{\varepsilon} \to D \left(u \psi\right)$ in $L^{s(t,x)}(\Omega_T)$ as $\varepsilon \to 0^+$.
\end{enumerate}
\end{thm}

\noindent The rest of this section is devoted to the proof of Theorem \ref{thm:approx_theorem}. \\

\noindent We start with the following decomposition of $\Omega$ which will be used throughout the whole paper.
\begin{lem}\label{lem:decomposition_Omega}
There exists $r>0$ and an open finite covering $\{\mathcal{B}^i_r\}_{i=1}^N$ of $\Omega$ by balls of radii $r$ such that if we define
$$
q_i(t) := \inf_{x \in \mathcal{B}_{2r}^i} s(t,x), \qquad r_i(t) := \sup_{x \in \mathcal{B}_{2r}^i} s(t,x), \qquad R_i(t) := q_i(t)\,\left(1+\frac{2}{d}\right)
$$
we have for all $i=1,\ldots, N$
$$
s_{\text{min}} \leq q_i(t) \leq s(t,x) \leq r_i(t) < R_i(t) 
\qquad \mbox{ on } (0,T) \times (\mathcal{B}^{i}_{2r} \cap \Omega)
$$
and
$$
R_i(t) - r_i(t) \geq \frac{s_{\text{min}}}{d}.
$$
\end{lem}
\begin{proof}
We cover $\Omega$ with balls $\mathcal{B}_r^i$ of equal radius $r$ and the only problem is to find radius $r$ satisfying assertions of the lemma. By \ref{ass:cont} in Assumption \ref{ass:exponent_cont_space} we can choose $r$ such that
$$
\sup_{x \in \mathcal{B}_{2r}^i \cap \Omega} s(t,x) - \inf_{x \in \mathcal{B}_{2r}^i \cap \Omega} s(t,x) \leq \frac{s_{\text{min}}}{d}.
$$
Since $s_{\text{min}} \leq q_i(t)$, the conclusion follows.
\end{proof}
\begin{Not}\label{not:zeta}
In what follows, we always consider the covering constructed in Lemma~\ref{lem:decomposition_Omega}. We also write $\zeta_i$, $i = 1,\ldots, N$ for the partition of unity related to the open covering $\{\mathcal{B}^{i}_{r}\}$ of $\Omega$, that is $\mbox{supp}\,\zeta_i \subset \mathcal{B}^{i}_{r}$ and $\sum_{i=1}^N \zeta_i(x) = 1$ for all $x\in \Omega$.
\end{Not}

\noindent Let us first observe that the fact $Du \in L^{s(t,x)}(\Omega_T)$ implies certain regularity properties.
\begin{lem}\label{lem:simple_integrability_lemma}
Suppose that Assumption \ref{ass:exponent_cont_space} holds true. Let $u$ be such that $Du \in L^{s(t,x)}(\Omega_T)$, $u \in L^1_t W^{1,1}_{0,x}$ and  $u\in L^{\infty}_t L^2_x$. Then, $u \in L^{q_i(t)}(0,T; W^{1,q_i(t)}(\mathcal{B}_{2r}^i)) \cap L^{s_{\text{min}}}(0,T;W^{1,s_{\text{min}}}_0(\Omega))$ for all $i=1,\ldots, N$. The norm of $u$ in these spaces depends only on
$
\|Du\|_{L^{s(t,x)}}, \|u\|_{L^{\infty}_t L^2_x}.
$
\end{lem}
\begin{proof}
As $q_i(t) \leq s(t,x)$ on $\mathcal{B}_{2r}^i$, we deduce $Du \in L^{q_i(t)}(0,T; L^{q_i(t)}(\mathcal{B}_{2r}^i))$. Then, the generalized K\"{o}rn inequality implies that (for fix $t$ and $i$)
$$
\|\nabla u\|_{L^{q_i(t)}(\mathcal{B}_{2r}^i)}\le C(r, s_{\text{min}}, s_{\text{max}})(\|D u\|_{L^{q_i(t)}(\mathcal{B}_{2r}^i)} + \|u\|_2).
$$
Then rasing the inequality to the $q_i(t)$ power, integrating over $t\in (0,T)$ and using the assumptions on $u$ we have the first part of the statement. The second statement can proved exactly in the same way using that $s_{\text{min}} \leq s(t,x)$ on $\Omega_T$ and the standard K\"{o}rn inequality for functions having zero trace.
\end{proof}

\noindent The most important tool is to approximate function $\xi \mapsto |\xi|^{s(t,x)}$ with functions independent of $x$ or $t$. This is obtained in the following lemmas.

\begin{lem}\label{lem:existence_of_minimizer}
Suppose that Assumption \ref{ass:exponent_cont_space} is satisfied. Then, for a.e. $t \in (0,T)$ and all balls $B_{\gamma}(x)$ such that $\overline{B_{\gamma}(x)} \cap \overline{\Omega}$ is nonempty, there exists $x^* \in \overline{B_{\gamma}(x)} \cap \overline{\Omega}$, $x^* = x^*(B_{\gamma}(x), t)$, such that for all $\xi$ with $|\xi| \geq 1$, we have
$$
\inf_{y \in \overline{B_{\gamma}(x)} \cap \overline{\Omega}} |\xi|^{s(t,y)}= |\xi|^{s(t,x^*)}.
$$
\end{lem}
\begin{rem}
Note carefully that the minimizing point $x^*$ is independent of $\xi$.
\end{rem}
\begin{proof}[Proof of Lemma \ref{lem:existence_of_minimizer}]
First, we note that for a.e. $t \in (0,T)$, the map $x \mapsto s(t,x)$ is continuous and so is the map $x \mapsto|\xi|^{s(t,x)}$ (for fixed $\xi$). Using compactness of $\overline{B_{\gamma}(x)} \cap \overline{\Omega}$ and $|\xi|>1$, we have
$$
\inf_{y \in \overline{B_{\gamma}(x)} \cap \overline{B}} |\xi|^{s(t,y)} =  |\xi|^{\inf_{y \in \overline{B_{\gamma}(x)} \cap \overline{\Omega}} s(t,y)}
$$
and we choose ${x}^*$ such that $\inf_{y \in \overline{B_{\gamma}(x)} \cap \overline{\Omega}} s(t,y) = s(t,x^*)$.
\end{proof}

\begin{lem}\label{lem:estimate_on_infimum}
Let $E>0$ be given. Then, there exists a constant $M = M(E)$, such that
$$
|\xi|^{s(t,y)} \leq M \, \inf_{z \in \overline{B_{\gamma}(x)} \cap \overline{\Omega}} |\xi|^{s(t,z)} \leq M \, |\xi|^{s(t,y)}
$$
for all balls $B_{\gamma}(x)$, all $y \in \overline{B_{\gamma}(x)} \cap \overline{\Omega}$, all $\xi \in [1,E\, \gamma^{-(d+1)}]$ and all $\gamma \in \left(0,\frac{1}{2}\right)$.
\end{lem}
\begin{proof}
Let $y_1, y_2 \in B_{\gamma}(x)$. As $|\xi| \geq 1$, we have
$$
\frac{|\xi|^{s(t,y_1)}}{|\xi|^{s(t,y_2)}} = |\xi|^{s(t,y_1) - s(t,y_2)} \leq |\xi|^{|s(t,y_1) - s(t,y_2)|}
$$
Using log-H\"older continuity (\ref{ass:exponent_cont_space}) in Assumption \ref{ass:exponent_cont_space}, we get
$$
|\xi|^{|s(t,y_1) - s(t,y_2)|}  \leq |\xi|^{-\frac{C}{\log|y_1-y_2|}} \leq \left(E\, \gamma^{-(d+1)}\right)^{-\frac{C}{\log \gamma}} = E^{-\frac{C}{\log \gamma}} \, \gamma^{\frac{C\,(d+1)}{\log \gamma}} \le E^{\frac{C}{\log 2}} \, e^{C\,(d+1)}.
$$
The conclusion follows from setting $M:= E^{\frac{C}{\log 2}} \, e^{C\,(d+1)}$.
\end{proof}

\begin{proof}[Proof of Theorem \ref{thm:approx_theorem}]
Properties \ref{thmitem:conv1} and \ref{thmitem:conv2} follow from the standard properties of the convolutions. To see \ref{thmitem:conv3}, we first estimate $ D\left(u^{\varepsilon} \psi\right)^{\varepsilon}$ in $L^{\infty}$ norm. By product rule,
$$
D\left(u^{\varepsilon} \psi\right)^{\varepsilon} = \left(D u^{\varepsilon} \,  \psi\right)^{\varepsilon} + \frac{\left(u^{\varepsilon} \otimes \nabla \psi+\nabla \psi \otimes u^{\varepsilon}\right)^{\varepsilon}}{2}.
$$
The first term on the right hand side can be estimated with the help of Young's inequality as
\begin{multline*}
\left\| \left(D u^{\varepsilon} \, \psi\right)^{\varepsilon} \right\|_{L^\infty_t L^{\infty}_x} \leq \left\|D u^{\varepsilon} \psi \right\|_{L^\infty_t L^{\infty}_x}
\leq \left\| \psi \right\|_{\infty} \, \left\|D u^{\varepsilon} \right\|_{L^\infty_t L^{\infty}_x}
\leq \\ \leq \left\| \psi \right\|_{\infty} \, \left\|u \right\|_{L^\infty_t L^{1}_x} \, \|D\eta_{\varepsilon}\|_{\infty}  \leq
\left\| \psi \right\|_{\infty} \, \left\|u \right\|_{L^\infty_t L^{1}_x} \, \|D\eta\|_{\infty} \, \frac{1}{\varepsilon^{d+1}}.
\end{multline*}
Similarly, we have for the second term
$$
\left\| \left(u^{\varepsilon} \otimes \nabla  \psi\right)^{\varepsilon} \right\|_{L^\infty_t L^{\infty}_x} \leq \left\|u^{\varepsilon} \otimes \nabla  \psi \right\|_{L^\infty_t L^{\infty}_x}
\leq \left\| D \psi \right\|_{\infty} \, \left\|u^{\varepsilon} \right\|_{L^\infty_t L^{\infty}_x}
\leq \left\| D \psi \right\|_{\infty} \, \left\| u \right\|_{L^\infty_t L^{1}_x} \, \|\eta\|_{\infty} \, \frac{1}{\varepsilon^d}.
$$
It follows that there exists a constant $E$ depending on $W^{1,\infty}$ norm of $\psi$, $L^{\infty}_t L^2_x$ norm of $u$, such that
\begin{equation}\label{cislo}
\| D\left(u^{\varepsilon} \psi\right)^{\varepsilon} \|_{\infty} \leq \frac{E}{\varepsilon^{d+1}}.
\end{equation}
Now, we estimate $\left|D\left(u^{\varepsilon} \psi\right)^{\varepsilon}\right|^{s(t,x)}$. Clearly, when $\left|D\left(u^{\varepsilon} \psi\right)^{\varepsilon}\right| \leq 1$, we have
\begin{equation}\label{eq:estimate_gradient_small_arg}
\left|D\left(u^{\varepsilon} \psi\right)^{\varepsilon}\right|^{s(t,x)} \leq 1.
\end{equation}
Suppose that $\left|D\left(u^{\varepsilon} \psi\right)^{\varepsilon}\right| \geq 1$. We fix $x \in \Omega$ and consider a ball $B_{3\,\varepsilon}(x)$. Then, from Lemmas \ref{lem:existence_of_minimizer},  \ref{lem:estimate_on_infimum} and the estimate \eqref{cislo} we obtain minimizing point $x^* \in B_{3\,\varepsilon}(x)$, such that
\begin{equation}\label{eq:estimate_gradient_large_arg}
\left|D\left(u^{\varepsilon} \psi\right)^{\varepsilon}\right|^{s(t,x)} \leq
M \, \left|D\left(u^{\varepsilon} \psi\right)^{\varepsilon}\right|^{s(t,x^*)}.
\end{equation}
Combining two estimates \eqref{eq:estimate_gradient_small_arg} and \eqref{eq:estimate_gradient_large_arg}, we deduce
\begin{equation}\label{eq:estimate_gradient_large_and_small}
\left|D\left(u^{\varepsilon} \psi\right)^{\varepsilon}\right|^{s(t,x)} \leq 1 + M \, \left|D\left(u^{\varepsilon} \psi\right)^{\varepsilon}\right|^{s(t,x^*)}.
\end{equation}
Next, the function $v \mapsto |v|^{s(t,x^*)}$ is convex. Therefore, Jensen's inequality implies
\begin{align*}
\left|D\left(u^{\varepsilon} \psi\right)^{\varepsilon}\right|^{s(t,x^*)}(t,x) &=
\left|\int_{B_{\varepsilon}(0)} D\left(u^{\varepsilon}(t,x-y) \, \psi(x-y)\right) \, \eta_{\varepsilon}(y) \diff y\right|^{s(t,x^*)}\\
&\leq
\int_{B_{\varepsilon}(0)} \left|D\left(u^{\varepsilon}(t,x-y) \, \psi(x-y)\right)\right|^{s(t,x^*)} \, \eta_{\varepsilon}(y) \diff y \\
&\le C(s_{\max})\, \int_{B_{\varepsilon}(0)} \left|D u^{\varepsilon}(t,x-y) \, \psi(x-y) \right|^{s(t,x^*)} \, \eta_{\varepsilon}(y) \diff y \, + \\
&\phantom{= }+ C(s_{\max}) \, \int_{B_{\varepsilon}(0)} \left| u^{\varepsilon}(t,x-y) \otimes  \nabla \psi(x-y) \right|^{s(t,x^*)} \, \eta_{\varepsilon}(y) \diff y \\
& \leq C(s_{\max})\, \|\psi\|_{\infty}^{s_{\text{max}}} \,\int_{B_{\varepsilon}(0)} \left|D u^{\varepsilon}(t,x-y) \right|^{s(t,x^*)} \, \eta_{\varepsilon}(y) \diff y \, + \\
& \phantom{\leq} \, +C(s_{\max}) \, \|\nabla \psi\|^{s_{\text{max}}} \,\int_{B_{\varepsilon}(0)} \, \left|u^{\varepsilon}(t,x-y) \right|^{s(t,x^*)} \, \eta_{\varepsilon}(y) \diff y
:= A+B.
\end{align*}
Concerning term $A$, we expand convolution, then we apply Jensen's inequality again and finally we observe that $x - y - z \in B_{3\,\varepsilon}(x)$, which allows us to apply Lemma \ref{lem:estimate_on_infimum} to get
\begin{equation}\label{eq:estimate_with_aeps_at_the_end}
\begin{split}
A &\leq C \, \| \psi \|_{\infty}^{s_{\text{max}}} \, \int_{B_{\varepsilon}(0)} \int_{B_{\varepsilon}(0)} \left|D u(t,x-y-z) \right|^{s(t,x^*)} \,   \eta_{\varepsilon}(y) \,  \eta_{\varepsilon}(z) \diff y \diff z\\
&\leq C \, \| \psi \|_{\infty}^{s_{\text{max}}} \,\int_{B_{\varepsilon}(0)} \int_{B_{\varepsilon}(0)} \left|D u(t,x-y-z) \right|^{s(t,x-y-z)} \,   \eta_{\varepsilon}(y) \,  \eta_{\varepsilon}(z) \diff y \diff z =: a_{\varepsilon}(t,x)
\end{split}
\end{equation}
Concerning term $B$, we proceed similarly, in the last step using partition of unity $\{\zeta_i\}$ from Notation~\ref{not:zeta}:
\begin{equation}\label{eq:estimate_with_beps_at_the_end}
\begin{split}
B &\leq C\, \|\nabla \psi\|_{\infty}^{s_{\text{max}}} \, \int_{B_{\varepsilon}(0)} \int_{B_{\varepsilon}(0)}  \left| u(t,x-y-z) \right|^{s(t,x^*)} \,   \eta_{\varepsilon}(y) \,  \eta_{\varepsilon}(z) \diff y \diff z\\
&\leq C\, \|\nabla \psi\|_{\infty}^{s_{\text{max}}} \, \int_{B_{\varepsilon}(0)} \int_{B_{\varepsilon}(0)}  \left(1 + \left| u(t,x-y-z) \right|^{s(t,x-y-z)}\right) \,   \eta_{\varepsilon}(y) \,  \eta_{\varepsilon}(z) \diff y \diff z\\
&\leq C\sum_{i=1}^N \zeta_i(x) \, \| \nabla \psi\|_{\infty}^{s_{\text{max}}} \, \int_{B_{\varepsilon}(0)} \int_{B_{\varepsilon}(0)}  \left(1 + \left| u(t,x-y-z) \right|^{r_i(t)}\right) \,   \eta_{\varepsilon}(y) \,  \eta_{\varepsilon}(z) \diff y \diff z\\
&=: C\sum_{i=1}^N b_{\varepsilon}^i(t,x).\phantom{\int_{B_{\varepsilon}(0)}}
\end{split}
\end{equation}
where in the last line we used that when $x \in \mathcal{B}^{i}_{r} \cap \Omega$, we have $s(t,x-y-z) \leq r_i(t)$ as long as $2\varepsilon < r$. Combining \eqref{eq:estimate_gradient_large_and_small}, \eqref{eq:estimate_with_aeps_at_the_end}, \eqref{eq:estimate_with_beps_at_the_end} we obtain
\begin{equation}\label{eq:uniform_integrability_gradient_in_L1}
0 \leq \left|D\left(u^{\varepsilon} \psi\right)^{\varepsilon}\right|^{s(t,x)} \leq C\left(1 + a_{\varepsilon}(t,x) + \sum_{i=1}^N b^i_{\varepsilon}(t,x)\right).
\end{equation}
Now, the map $(t,x) \mapsto |Du(t,x)|^{s(t,x)} \in L^1(\Omega_T)$ so that due to Lemma \ref{lem:double_moll_conv}, $a_{\varepsilon}$ is convergent in $L^1(\Omega_T)$. Similarly, Lemma \ref{lem:simple_integrability_lemma} together with interpolation result from Lemma \ref{thm:interpolation} shows that the map $(t,x) \mapsto |u(t,x)|^{r_i(t)} \in L^1((0,T)\times \mathcal{B}^{i}_{2r})$. Thanks to Lemma \ref{lem:double_moll_conv}, $b_{\varepsilon}^i$ is also convergent in $L^1((0,T)\times \mathcal{B}^{i}_{r})$ so that taking into account the supports of $\zeta_i$, $b_{\varepsilon}^i$ is also convergent in $L^1(\Omega_T)$. Now, from \eqref{eq:uniform_integrability_gradient_in_L1}, we deduce that the map $(t,x) \mapsto \left|D\left(u^{\varepsilon} \psi\right)^{\varepsilon}\right|^{s(t,x)}$ is uniformly integrable. Together with pointwise convergence as $\varepsilon \to 0$, this is sufficient to conclude $D\left(u^{\varepsilon} \psi\right)^{\varepsilon} \to D(u\, \psi)$ in $L^{s(t,x)}(\Omega_T)$.
\end{proof}

\section{Local energy equality}\label{section:4}

\noindent First, we  discuss a general abstract identity
\begin{align}\label{eq:gen_identity_def}
    \int_{\Omega_T}-u \cdot \p_t \phi - u\otimes u:\nabla\phi + (\alpha +\theta\, \beta) : D\phi  \diff x \diff t = \int_{\Omega}u_0(x) \cdot\phi(0, x) \diff x + \int_{\Omega_T}f  \cdot\phi \diff x \diff t
\end{align}
which is required to be satisfied for any vector-valued $\phi \in C^\infty_c([0, T) \times \Omega)$ fulfilling $\DIV \phi = 0$ in $\Omega_T$. Here, $\alpha, \beta: \Omega_T \to \symm$, $\theta \geq 0$ and the term $\theta \, \beta$ can be seen as a regularizing term. In case $\theta = 0$ we are inspired by our setting and assume only
\begin{equation}\label{eq:regularity_estimates_u}
u \in L^\infty_t L^2_x \cap L^{R_i(t)}((0,T)\times \mathcal{B}^i_{2r}) \cap L^{s_0}_{t,x}, \qquad Du \in L^{s(t,x)}(\Omega_T)
\end{equation}
\begin{equation}\label{eq:regularity_estimates_not_u_theta_0}
\alpha \in L^{s'(t,x)}(\Omega_T) \mbox{ is a symmetric matrix}, \qquad f\in L^{1}_t L^2_x,
\end{equation}
where, we recall \eqref{s0}, i.e., that  $s_0 = 3 + \frac{2}{d}$. In case $\theta > 0$, we can additionally assume that
\begin{equation}\label{eq:regularity_estimates_not_u_2}
\beta \in L^{s_{\max}'}(\Omega_T) \mbox{ is a symmetric matrix}, \quad Du \in L^{s_{\max}}(\Omega_T), \qquad \theta > 0.
\end{equation}
Nevertheless, the uniform estimates can be obtained only in terms of~\eqref{eq:regularity_estimates_u}--\eqref{eq:regularity_estimates_not_u_theta_0} as we expect to lose~\eqref{eq:regularity_estimates_not_u_2} when $\theta \to 0$. This becomes more visible in the next section when we apply results obtained for the particular $\alpha$ and $\beta$. Notice also here, that $s_0\le R_i(t)$, which follows from the definition of $R_i(t)$ in Lemma~\ref{lem:decomposition_Omega}, the definition of $s_0$ in \eqref{s0} and the assumption $s_{\text{min}}\ge \frac{3d+2}{d+2}$. \\

\noindent Our first target is to transform this identity to be satisfied by all test functions, not necessarily divergence-free. Here, the main difficulty is that we consider the problem with Dirichlet boundary condition. Hence, we cannot apply Hodge decomposition theorem. Instead, we apply Ne\v{c}as theorem and use the method of harmonic pressure cf. \cite{MR2305828,MR1602949,MR2668872}.\\

\noindent First, we extend all functions to $\R^d$ with zero and we define $p_1^i$, $p_2^i$, $p_3$, $p_4$ as the unique functions satisfying for almost all $t\in (0,T)$
\begin{align}
 -\Delta p_{1}^i &= \DIV \DIV (\alpha \, \zeta_i) \mbox{ in } {\R^d}, \quad   &&{p_1^i(t,\cdot) \in L^{r_i'(t)}(\R^d)}, \label{eq:defp1A}\\
 \Delta p_{2}^i &= \DIV \DIV (u\otimes u \, \zeta_i) \mbox{ in } \R^d, \quad   &&{p_2^i(t,\cdot) \in L^{R_i(t)/2}(\R^d)}, \label{eq:defp2A}\\
 -\Delta p_3 &= \DIV f \mbox{ in } \R^d, \quad  &&p_3(t,\cdot) \in W^{1,2}_{\textrm{loc}}(\R^d), \label{eq:defp3} \\
 -\Delta p_4 &= \DIV \DIV (\theta\, \beta) \mbox{ in } {\R^d}, \quad  &&{p_4(t,\cdot) \in L^{s'_{\max}}(\R^d)}. \label{eq:defp4}
\end{align}

\noindent These functions have to be defined globally as we expect them to have only Lebesgue regularity so their trace is not well-defined. The main result of this section reads:
\begin{thm}\label{thm:local_energy_equality}
Let $u$ be a solution of \eqref{eq:gen_identity_def}. Then, with $p_i$ as above,
there exists a uniquely determined function $p_h \in L^{\infty}(0,T; L^{s'_{\max}}(\Omega))$ (up to condition $\int_{\Omega} p_h(t,x) \diff x  = 0$ for almost all $t\in (0,T)$), such that for all $\varphi \in C^1_c([0,T)\times \Omega)$
\begin{equation}\label{eq:distributional_id_with_ph}
\begin{split}
-\int_{\Omega_T} &(u+\nabla p_h)\cdot \partial_t\varphi +(- u \otimes u + \alpha + \theta \beta) : \nabla \varphi \diff x \diff t-\int_{\Omega} u_0(x)\cdot \varphi(0,x)  \diff x \\
&= \int_{\Omega_T} f\, \cdot \varphi  - \sum_{i=1}^N (p_1^i + p_2^i) \,\DIV \varphi -  (p_3 + p_4) \, \DIV \varphi \diff x \diff t
\end{split}
\end{equation}
Moreover, $p_h$ is harmonic and so, it is locally smooth in the spatial variable, i.e., it satisfies
\begin{equation}\label{eq:pressure_is_harmonic}
\Delta p_h = 0 \mbox{ in } \Omega \mbox{ for a.e. } t \in [0,T].
\end{equation}
In addition, we have the following estimate valid for all $\Omega' \Subset \Omega$
\begin{equation}\label{eq:thm_estimate_on_p_Hol}
\begin{split}
\|p_h\|_{L^{\infty}_t L^{s'_{\max}}_x}+ \|p_h\|_{L^{\infty}(0,T;C^k(\Omega'))} \leq C,
\end{split}
\end{equation}
where the constant $C$ depends on $k$, $\Omega'$, $T$ and the norms $\|\alpha\|_{L^{s'(t,x)}_{t,x}}$, $\|Du\|_{L^{s(t,x)}_{t,x}}$, $\|\theta \, \beta\|_{L^{s'_{\max}}_{t,x}}$,
 $\|f\|_{L^{1}_t L^2_x}$, $\|u\|_{L^{\infty}_t L^2_x}$ and $\|u_0\|_{L^{2}_x}$. Finally, the following local energy equality holds: for any $\psi\in C^\infty_c(\Omega)$ and for a.e. $t \in (0,T)$
\begin{equation}\label{eq:local_energy_equality}
\begin{split}
    &\frac{1}{2}\int_\Omega |u(t, x) + \nabla p_h(t, x)|^2 \, \psi(x) \diff x + \int_{0}^t \int_\Omega  (\alpha + \theta \beta):D(\psi(x)(u + \nabla p_h)(\tau,x)) \diff x \diff \tau = \\
    &= \frac{1}{2}\int_{\Omega}  |u_0(x)|^2 \, \psi(x) \diff x + \int_{0}^t \int_\Omega  (u\otimes u) : \nabla (\psi (u + \nabla p_h))\diff x\diff \tau + \\
    &\phantom{=} +  \int_{0}^t\int_{\Omega}f(u + \nabla p_h)\psi\diff x \diff \tau - \int_{0}^t \int_\Omega  \left(\sum_{i=1}^N (p_1^i + p_2^i) + p_3 + p_4\right)\,(u + \nabla p_h )\nabla\psi \diff x \diff \tau
\end{split}
\end{equation}
\end{thm}

\noindent The rest of this section is devoted to the proof of Theorem \ref{thm:local_energy_equality}. First, we establish regularity of $p_1^i$, $p_2^i$, $p_3$, $p_4$.

\begin{lem}\label{lem:integrability_of_pi}
There exists uniquely determined $p_1^i$, $p_2^i$, $p_3$, $p_4$ satisfying \eqref{eq:defp1A}--\eqref{eq:defp4}. Moreover, there exists a constant depending only on $T$ and ${\Omega}$ such that
\begin{align*}
&\|p_{1}^i\|_{L^{r_i'(t)}_{t,x}} \leq C \|\alpha \, \zeta_i\|_{L^{r_i'(t)}_{t,x}},    &&\|p_{2}^i\|_{L^{R_i(t)/2}_{t,x}} \leq C \|u \, \sqrt{\zeta_i}\|_{L^{R_i(t)}_{t,x}}^2,\\
&\|p_{1}^i\|_{L^{{s'_{\max}}}_{t,x}} \leq C\, \| \alpha \|_{L^{{s'_{\max}}}_{t,x}},
&&\|p_{2}^i\|_{L^{{s_{0}}/2}_{t,x}} \leq C\, \| u \|^2_{L^{{s_{0}}}_{t,x}}\\
&\|p_3\|_{L^1_t W^{1,2}_x} \leq C\, \|f\|_{L^1_1L^2_x},  &&\|p_4\|_{L^{{s'_{\max}}}_{t,x}} \leq C\, \| \theta\,\beta\|_{L^{{s'_{\max}}}_{t,x}}
\end{align*}
Moreover, for each bounded $\Omega' \Subset \R^d \setminus \mathcal{B}^i_{r}$ we have
$$
\| p_{1}^i \|_{L^{s'_{\max}}_t(0,T; C^{k}_x(\Omega'))} \leq C\left(T,k,\Omega'\right) \| \alpha \|_{L^{{s'_{\max}}}_{t,x}},
$$
$$
\| p_{2}^i \|_{L^{s_{0}/2}_t(0,T; C^{k}_x(\Omega'))} \leq C\left(T,k,\Omega'\right) \| u \|^2_{L^{{s_{0}}}_{t,x}}.
$$
\end{lem}
\begin{proof}
For $p_{1}^i$, $p_{2}^i$ and $p_4$ this follows immiedately from Theorem \ref{thm:poissonLp} because $\alpha \, \zeta_i \in L^{r_i'(t)}_{t,x}$, $\alpha \in L^{s'_{\max}}_{t,x}$, $u \otimes u \, \zeta_i \in L^{R_i(t)/2}_{t,x}$ (with a norm controlled by $\|u \, \sqrt{\zeta_i}\|^2_{L^{R_i(t)}_{t,x}}$), $u \otimes u \in L^{s_0/2}_{t,x}$ and $\beta \in L^{s_{\max}'}_{t,x}$. To see the result for $p_3$, we fix $t \in (0,T)$ and solve the problem
$$
\int_{\mathbb{R}^d}\nabla p_3 \cdot \nabla \varphi = -\int_{\mathbb{R}^d}f\cdot \nabla \varphi \quad \textrm{ for all } \varphi\in C_c(\mathbb{R}^d)
$$
with the assumption $p_3(x)\to 0$ as $|x|\to \infty$.
Therefore, the classical theory implies there exists a uniquely defined $p_3(t,x)$ fulfilling $\|\nabla p_3(t)\|_{L_x^2} \le \|f(t)\|_{L^2_x}$. Consequently the statement of the lemma for $p_3$ follows from the embedding theorem. \\

\noindent Finally, we obtain better estimates for $p_{1}^i$ and $p_{2}^i$. We observe that these functions are harmonic in $\R^d \setminus \mathcal{B}^i_{r}$. Therefore, by Weyl's lemma, $p_{1}^i$ and $p_{2}^i$ are smooth in the spatial variable. To obtain uniform local estimate, we first consider $p_{1}^i$, we fix $\Omega' \Subset \R^d \setminus \mathcal{B}^i_{r}$ and apply \cite[Theorem 4.2]{MR1616087} to deduce
$$
\| p_{1}^i \|_{L^{s'_{\max}}_t(0,T; W^{2,s'_{\max}}_x(\Omega'))} \leq C\left(T, \Omega'\right)\, \| p_{1}^i \|_{L^{s'_{\max}}_{t,x}} \leq C\left(T, \Omega'\right) \| \alpha \|_{L^{{s'_{\max}}}_{t,x}}.
$$
As $\Delta p_{1}^i = 0$ in $\R^d \setminus \mathcal{B}^i_{r}$, we may iterate to get
$$
\| p_{1}^i \|_{L^{s'_{\max}}_t(0,T; W^{2k,s'_{\max}}_x(\Omega'))} \leq C\left(T,k,\Omega'\right) \| \alpha \|_{L^{{s'_{\max}}}_{t,x}}
$$
and this concludes the proof by the Sobolev embedding. The proof for $p_{2}^i$ is completely analogous.
\end{proof}

\begin{lem}\label{lem:pressure}
There exists uniquely determined $p_h$ satisfying \eqref{eq:distributional_id_with_ph}  and $\int_{\Omega} p_h(t,x) \diff x = 0$. Moreover, \eqref{eq:pressure_is_harmonic} and \eqref{eq:thm_estimate_on_p_Hol} hold true.
\end{lem}
\begin{proof}
We divide the proof for several steps.\\

\noindent \underline{Existence of $p_h$.} For fixed time $t>0$, we consider functional
\begin{multline*}
\mathcal{G}(\varphi):= \int_{\Omega} (u(t,x)- u_0(x)) \cdot \varphi(x) + \int_0^t \int_{\Omega} (- u \otimes u + \alpha + \theta \beta) : \nabla \varphi - \int_0^t \int_{\Omega} f \cdot \varphi \,+\\
+ \sum_{i=1}^N\int_0^t \int_{\Omega} (p_1^i + p_2^i) \, \DIV \varphi+  \int_0^t \int_{\Omega} (p_3 + p_4) \, \DIV \varphi
\end{multline*}
acting on vector-valued functions $\varphi: \Omega \to \R^d$. Now, we establish regularity of functional $\mathcal{G}$ by estimating terms appearing in its definition. First, thanks to \eqref{eq:regularity_estimates_u}--\eqref{eq:regularity_estimates_not_u_2} we have
$$
\left|\int_{\Omega} (u(t,x)- u_0(x)) \cdot \varphi(x)\right| \leq \left( \|u\|_{L^{\infty}_t L^2_x} + \|u_0\|_{L^{2}_x} \right) \| \varphi \|_{L^{2}_x},
$$
\begin{multline*}
\left|\int_0^t \int_{\Omega} (- u \otimes u + \alpha + \theta \beta) : \nabla \varphi\right| \leq \\
\leq C(T)\,\left(
\|u\|_{L^{s_{0}}_{t,x}}^2 \| \nabla \varphi\|_{L^{(s_{0}/2)'}_x} +
\left(\| \alpha\|_{L^{s'_{\max}}_{t,x}} + \| \theta \beta\|_{L^{s'_{\max}}_{t,x}}  \right)\, \| \nabla \varphi \|_{L^{s_{\max}}_x} \right).
\end{multline*}
Here, the constant $C(T)$ comes from applying H\"{o}lder's inequality in the time variable. Next,
$$
\left| \int_0^t \int_{\Omega} f\cdot \varphi \right| \leq C(T) \, \|f\|_{L^{1}_tL^2_x} \, \|\varphi\|_{L^{2}_{x}}.
$$
For the terms with $p_1^i$, $p_2^i$, $p_3$, $p_4$ we apply Lemma \ref{lem:integrability_of_pi} to obtain
$$
\begin{aligned}
\left|\int_0^t \int_{\Omega} p_1^i \,\DIV \varphi \right| &\leq  C(T, \|\alpha  \|_{L^{s'_{\max}}_{t,x}}) \, \|\nabla \varphi\|_{L^{s_{\max}}_x},\\
\left|\int_0^t \int_{\Omega} p_2^i \, \DIV \varphi \right| &\leq C(T, \|u\|_{L^{s_0}_{t,x}}^2)\, \| \nabla \varphi \|_{L^{\left(s_{0}/2\right)'}_x},\\
\left|\int_0^t \int_{\Omega} p_3 \, \DIV \varphi \right| &\leq C(T, \|f\|_{L^{1}_tL^2_x})\, \| \nabla \varphi \|_{L^{2}_x},\\
\left|\int_0^t \int_{\Omega} p_4 \, \DIV \varphi \right| &\leq C(T, \|\theta \, \beta\|_{L^{s'_{\max}}_{t,x}})\, \| \nabla \varphi \|_{L^{s_{\max}}_x}.
\end{aligned}
$$
We conclude that $\mathcal{G}$ is a bounded functional on $W^{1,s_{\max}}_0(\Omega)$ because $(s_0/2)' = \frac{3d+2}{d+2} \leq s_{\max}$. Moreover, if $\DIV \varphi = 0$, we obtain from \eqref{eq:gen_identity_def} that $\mathcal{G}(\varphi) = 0$. Hence, applying the Ne\v{c}as theorem, we know that there exists a distribution $p_h(t)$ fulfilling $\mathcal{G}(\varphi)=-\langle \nabla p_h(t), \varphi\rangle$ in sense of distributions. Using then the above estimates and the Ne\v{c}as theorem about negative norms, cf. Lemma~\ref{lem:deRham}, we get that for a.e. $t \in (0,T)$ there exists uniquely (up to function depending only on time) defined function $p_h(t,x)$ fulfilling
\begin{equation}\label{tlak1}
\begin{split}
\int_{\Omega} &(u(t,x)- u_0(x)) \, \varphi(x) + \int_0^t \int_{\Omega} (- u \otimes u + \alpha + \theta \beta) \cdot \nabla \varphi +\int_0^t \int_{\Omega} f\, \cdot \varphi \, +  \\ &+ \sum_{i=1}^N \int_0^t \int_{\Omega} (p_1^i + p_2^i) \cdot \DIV \varphi + \int_0^t \int_{\Omega} (p_3 + p_4) \cdot \DIV \varphi - \int_{\Omega} p_h(t,x) \, \DIV \varphi = 0.
\end{split}
\end{equation}
Moreover, if we require that $\int_{\Omega}p_h(t,x)=0$ for $t \in (0,T)$, then  $\|p_h(t,\cdot)\|_{L^{s'_{\max}}_x} \leq \widetilde{C}$, where $\widetilde{C}$ is a constant
$$
\widetilde{C}\left(T,
\|Du\|_{L^{s(t,x)}_{t,x}}, \|\alpha\|_{L^{s(t,x)}_{t,x}},
\|\theta \, \beta\|_{L^{s'_{\max}}_{t,x}},
\|f\|_{L^{1}_tL^2_x},
 \|u\|_{L^{\infty}_t L^2_x}, \|u_0\|_{L^{2}_x}\right).
$$
Here, we used the fact that
$\| u\|_{L^{s_0}_{t,x}}\le C(\|Du\|_{L^{s(t,x)}_{t,x}},\|u\|_{L^{\infty}_t L^2_x})$. Furthermore, we included $\|\alpha\|_{L^{s'_{\max}}_{t,x}}$ into $\|\alpha\|_{L^{s(t,x)}_{t,x}}$. Taking supremum over all $t \in (0,T)$, we deduce that $\|p_h\|_{L^{\infty}_t L^{s'_{\max}}_x} \leq \widetilde{C}$. \\

\noindent \underline{Function $p_h$ is harmonic.} Let $\varphi = \nabla \phi$ where $\phi \in C_c^{\infty}(\Omega)$. From \eqref{eq:defp1A}--\eqref{eq:defp4} we deduce
\begin{align*}
&\sum_{i=1}^N \int_0^t \int_{\Omega} (p_1^i + p_2^i) \, \DIV \varphi + \int_0^t \int_{\Omega} (p_3 + p_4) \, \DIV \varphi= \\
&\quad = \sum_{i=1}^N \int_0^t \int_{\Omega} (p_1^i + p_2^i) \, \Delta \phi + \int_0^t \int_{\Omega} (p_3 + p_4) \, \Delta \phi
=  \int_0^t \int_{\Omega} \left(-\alpha - \theta \beta +  u\otimes u \right) : \nabla^2 \phi +  f \cdot \nabla \phi
\end{align*}
because $\sum_{i=1}^N \zeta_i = 1$ on $\Omega$. Furthermore, the incompressibility condition yields
$$
\int_{\Omega} (u(t,x)- u_0(x)) \, \varphi(x) \diff x = 0.
$$
Therefore,
$$
\int_{\Omega} p_h(t,x) \, \Delta \phi(x) \diff x = 0
$$
so that $\Delta p_h(t,x) = 0$ in $\Omega$  in the sense of distributions.\\

\noindent \underline{Estimate on $p_h$.} By Weyl's lemma, $p_h(t,x)$ is smooth in the spatial variable. To obtain uniform local estimate, we fix $\Omega' \Subset \Omega$ and apply \cite[Theorem 4.2]{MR1616087} to equation  to deduce
$$
\| p_h \|_{L^{\infty}_t(0,T; W^{2,s'_{\max}}_x(\Omega'))} \leq C\, \| p_h \|_{L^{\infty}_t(0,T; L^{s'_{\max}}_x(\Omega))} \leq C(\widetilde{C}, \Omega').
$$
As the (RHS) of equation $\Delta p_h = 0$ is zero, we may iterate to prove
$$
\| p_h \|_{L^{\infty}_t(0,T; W^{2k,s'_{\max}}_x(\Omega'))} \leq C(\widetilde{C}, \Omega', k).
$$
The conclusion follows by Sobolev embeddings.

\noindent \underline{Weak formulation for $p_h$.}
Multiplying \eqref{tlak1} by an arbitrary $\partial_t \tau$, where $\tau \in C^1_c[0,T)$, integrating over $t\in (0,T)$ and using integration by parts with respect to time and space, using also the fact that $p_h$ is smooth in the interior of $\Omega$, we deduce with the help of \eqref{tlak1} that for all smooth $\psi \in C_c(\Omega)$
\begin{equation*}
\begin{split}
\int_{\Omega_T} &\partial_t\tau\,(u+\nabla p_h)\cdot \varphi =\int_{\Omega_T} \partial_t\tau\,u\cdot \varphi -\int_{\Omega_T}\partial_t\tau\,p_h \DIV\varphi \\
&=\int_{\Omega_T} \partial_t \tau\,u_0\cdot \varphi  -\int_0^T \left(\partial_t \tau \int_0^t \int_{\Omega} (- u \otimes u + \alpha + \theta \beta) : \nabla \varphi +\int_0^t \int_{\Omega} f\, \cdot \varphi \right) \diff t \, +  \\
&- \sum_{i=1}^N \int_0^T \partial_t \tau \left(\int_0^t \int_{\Omega} (p_1^i + p_2^i) \,\DIV \varphi + \int_0^t \int_{\Omega} (p_3 + p_4) \, \DIV \varphi\right)\diff  t\\
&=-\int_{\Omega} u_0(x)\cdot \varphi(x) \tau(0) \diff x +\int_{\Omega_T} \tau\,(- u \otimes u + \alpha + \theta \beta) : \nabla \varphi + \tau\,f\, \cdot \varphi   \, +  \\
&+\sum_{i=1}^N \int_{\Omega_T}\tau\,(p_1^i + p_2^i) \,\DIV \varphi\,  +  \tau\,(p_3 + p_4) \, \DIV \varphi 
\end{split}
\end{equation*}
The relation \eqref{eq:distributional_id_with_ph} then follows directly.
\end{proof}

\noindent To prove Theorem \ref{thm:local_energy_equality}, we need to extend the equation for negative times in a way that keeps in mind the initial condition $u_0$. We also include here a simple yet important regularity assertion: a solution mollified in spatial variable gains Sobolev regularity in time.

\begin{lem}\label{lem:extension_by_0_u}
Let $u$ be as in \eqref{eq:distributional_id_with_ph} and let us extend $u$ to $\overline{u}$ by the following:
\begin{align*}
    \overline{u}(t,x) = \begin{cases}
    0 &\text{ when }t > T,\\
    u(t, x) &\text{ when }t\in (0, T],\\
    u_0(x) &\text{ when }t \leq 0.
    \end{cases}
\end{align*}
In addition, let $\overline{f}$, $\overline{p^i_1}$, $\overline{p^i_2}$, $\overline{p_3}$, $\overline{p_4}$, $\overline{p_h}$, $\overline{u\otimes u}$, $\overline{\alpha}$, $\theta \overline{\beta}$ denote the extension by $0$ outside of the time interval $(0,T)$ for the quantities  $f$, $p^i_1$, $p^i_2$, $p_3$, $p_4$, $p_h$, $u\otimes u$, $\alpha$, $\theta \beta$ respectively. Then
\begin{equation}\label{eq:weak_equality_extensions}
\begin{split}
    \int_{-T}^T&\int_{\Omega}-(\overline{u} + \nabla \overline{p_h}) \cdot \p_t \phi - \overline{u\otimes u}:\nabla\phi + (\overline{\alpha} + \theta \overline{\beta}): \nabla \phi \diff x \diff t = \\
    &= -\sum_{i = 1}^N \int_{-T}^T\int_{\Omega}(\overline{p_1^i}+\overline{p_2^i})\,\DIV{\phi}\diff x \diff t - \int_{-T}^T\int_{\Omega}(\overline{p_3}+\overline{p_4})\,\DIV{\phi}\diff x \diff t+ \int_{-T}^T\int_{\Omega}\overline{f} \cdot \phi\diff x \diff t
\end{split}
\end{equation}
holds for any $\phi\in C^\infty_c((-T, T) \times \Omega)$. Moreover,
\begin{equation}
\label{tdp}
\p_t(\overline{u}+\nabla\overline{p_h})\in L^1(-T, T; (W^{1,s_{\max}}_0(\Omega))^*).
\end{equation}
\end{lem}
\begin{proof}
From \eqref{eq:distributional_id_with_ph} we obtain (using also the zero extensions)
\begin{align*}
    &\int_{-T}^T\int_{\Omega}-(\overline{u} + \nabla\overline{p_h})\,\p_t \phi - \overline{u\otimes u}:\nabla\phi + (\overline{\alpha} + \theta\, \overline{\beta}):\nabla\phi \diff x \diff t =\\
    & \qquad \qquad = \int_{-T}^0\int_\Omega-u_0\,\p_t \phi\diff x\diff t + \int_0^T\int_{\Omega}-(u + \nabla p_h)\,\p_t \phi - u\otimes u : \nabla\phi + (\alpha + \theta \beta): D\phi \diff x\diff t\\
    & \qquad \qquad= -\int_{\Omega}u_0(x)\,\phi(0,x)\diff x + \int_0^T\int_{\Omega}-(u + \nabla p_h)\,\p_t \phi - u\otimes u : \nabla\phi + (\alpha + \theta \beta): D\phi \diff x\diff t\\
    & \qquad \qquad= -\sum_{i=1}^N \int_{\Omega_T} (p_1^i + p_2^i)\,\DIV\phi\diff x\diff t - \int_{\Omega_T} (p_3 + p_4)\,\DIV\phi\diff x\diff t  + \int_{\Omega_T}f\,\phi\diff x\diff t\\
    & \qquad \qquad= -\sum_{i=1}^N \int_{\Omega_T} (\overline{p_1^i} + \overline{p_2^i})\,\DIV\phi\diff x\diff t - \int_{\Omega_T} (\overline{p_3} + \overline{p_4})\,\DIV\phi\diff x\diff t  + \int_{-T}^T\int_{\Omega}\overline{f}\,\phi \diff x\diff t,
\end{align*}
which completes the proof of \eqref{eq:weak_equality_extensions}. The assertion \eqref{tdp} follows directly from  \eqref{eq:weak_equality_extensions} and the assumptions on data.
\end{proof}

\begin{proof}[Proof of Theorem \ref{thm:local_energy_equality}]
Most of the statements of Theorem~\ref{thm:local_energy_equality} have been proven in Lemma~\ref{lem:integrability_of_pi}--\ref{lem:extension_by_0_u}. The only missing, but essential point, is the energy equality  \eqref{eq:local_energy_equality}. Hence, we focus on it in what follows.
We use the following approximation of the indicator function $\mathbf{1_{(\eta, \varrho)}}$:
\begin{align}\label{eq:def_fun_gamma}
\gamma_{\eta, \varrho}^{\tau}(t) = \left\{\begin{array}{ll}
     0\text{ for } t\leq \eta - \tau \text{ or } t\geq \varrho + \tau \\
     1 \text{ for } \eta \leq t \leq \varrho\\
     \text{affine  for }t\in [\eta - \tau, \eta]\cup[\varrho, \varrho + \tau]
\end{array}
\right.
\end{align}
Here, $\eta$ might be both negative and positive. When $\eta = 0$, we will mean the function
\begin{align*}
\gamma_{0, \varrho}^{\tau}(t) = \left\{\begin{array}{ll}
     0\text{ on } [\varrho + \tau, 1] \\
     1 \text{ on } [0, \varrho]\\
     \text{affine  otherwise}
\end{array}
\right.
\end{align*}
as the needed approximations.\\

\noindent For $\varrho, \eta \in (0,T)$ and $\delta, \tau, \varepsilon$ nonnegative and sufficiently small we define
$$
\phi_{\eta, \varrho}^{\delta, \tau, \eps}(t, x) = (\mathcal{R}^\delta(u^\eps(t, x)+\nabla p_h^{\eps}(t,x))\,\psi(x)\,\gamma_{-\eta, \varrho}^\tau(t)))^\eps
$$
and we use it as a test function in \eqref{eq:weak_equality_extensions}. We recall here that mollification operator $\mathcal{R}^{\delta}$ is defined in Definition \ref{res:mol_in_ti}. We obtain five different parts and we will study their limits as $\tau\to 0$, $\delta\to 0$ and $\varepsilon \to 0$ separately.\\

\noindent \underline{Term $(\overline{u} + \nabla \overline{p_h})\cdot \p_t \phi$.}
First, thanks to \eqref{tdp}, we can deduce for the distributional derivative that $\p_t(\overline{u}+\nabla\overline{p_h})\in L^{s'_{\max}}(0,T; (W_0^{1,s_{\max}}(\Omega))^*)$, consequently, mollification with respect to the spatial variable leads to the fact
$$
\p_t(\overline{u}+\nabla\overline{p_h})^\eps\in L^1(-T, T; L^1(\Omega')) \qquad \textrm{for any }\Omega' \Subset \Omega.
$$
Therefore, we can use integration by parts and Fubini theorem applied to mollifier to deduce
\begin{align*}
-&\int_{-T}^T \int_{\Omega}(\overline{u} + \nabla \overline{p_h})\cdot  \p_t \phi \diff x \diff t\\
&=-\int_{-T}^T \int_{\Omega}(\overline{u} + \nabla \overline{p_h})\cdot \p_t (\mathcal{R}^\delta(u^\eps(t, x)+\nabla p_h^{\eps}(t,x))\,\psi(x)\,\gamma_{-\eta, \varrho}^\tau(t)))^\eps \diff x \diff t\\
&=-\int_{-T}^T \int_{\Omega}(\overline{u} + \nabla \overline{p_h})^{\eps}\cdot \p_t (\mathcal{R}^\delta(u^\eps(t, x)+\nabla p_h^{\eps}(t,x))\,\psi(x)\,\gamma_{-\eta, \varrho}^\tau(t))) \diff x \diff t\\
&=\int_{-T}^T\int_{\Omega} \partial_t(\overline{u} + \nabla \overline{p_h})^{\varepsilon}(t,x) \cdot \mathcal{R}^\delta(\overline{u}^\eps(t, x)+ \nabla \overline{p_h}^{\eps})\,\psi(x)\,\gamma_{-\eta, \varrho}^\tau(t)) \diff x \diff t.
\end{align*}
We observe that $u\in L^\infty_t L^2_x$ and $p_h\in L^\infty_t L^\infty_{x, loc}$ so $(\overline{u} + \nabla \overline{p_h})^\eps \in L^\infty_t L^\infty_{x, loc}$. Moreover, $\partial_t(\overline{u} + \nabla \overline{p_h})^{\varepsilon}(t,x) \in L^{1}_t L^1_x$ and so  we can apply Lebesgue's dominated convergence to converge with $\tau\to 0$, $\delta\to 0$ and obtain
\begin{align*}
\to \int_{-\eta}^\varrho\int_{\Omega} \p_t (\overline{u} + \nabla \overline{p_h})^\eps(t,x) \cdot  (\overline{u} + \nabla \overline{p_h})^\eps(t,x)\,\psi(x)\diff x \diff t.
\end{align*}
Now since $(\overline{u} + \nabla \overline{p_h})^\eps \in L^\infty_t L^\infty_{x, loc}$, we may apply chain rule for Sobolev functions similarly as in \cite[Lemma 2.14]{bulicek2021parabolic} to obtain
\begin{align*}
    \int_{-\eta}^\varrho\int_{\Omega}(\overline{u} + \nabla \overline{p_h})^\eps(t,x)\,\p_t (\overline{u} + \nabla \overline{p_h})^\eps(t,x)\,\psi(x)\diff x \diff t = \frac{1}{2}\int_{-\eta}^\varrho\int_{\Omega}\p_t|(\overline{u} + \nabla \overline{p_h})^\eps(t,x)|^2\,\psi(x)\diff x \diff t.
\end{align*}
By the absolute continuity of Sobolev functions on lines, we get that for all $\eta, \varrho \in (0, T)$
\begin{equation*}
\begin{split}
    &\int_{-\eta}^\varrho\int_{\Omega}\p_t|(\overline{u} + \nabla \overline{p_h})^\eps(t,x)|^2\,\psi(x)\diff x \diff t \\
    &= \int_{\Omega} |(\overline{u} + \nabla \overline{p_h})^\eps(\varrho, x)|^2 \, \psi(x) \diff x - \int_{\Omega}|(\overline{u} + \nabla \overline{p_h})^\eps(-\eta, x)|^2\,\psi(x) \diff x\\
    &= \int_{\Omega} |(u + \nabla p_h)^\eps(\varrho, x)|^2 \, \psi(x) \diff x - \int_{\Omega}|u_0^\eps(x)|^2\,\psi(x) \diff x.
\end{split}
\end{equation*}
Recall, we used the fact that $\overline{p_h}(t,x)=0$ and $u(t,x)=u_0(x)$ for negative $t$'s.
Now, we want to let $\varepsilon \to 0$. Note that for a.e. $\varrho \in (0,T)$ we have $u(\varrho,x) + \nabla p_h(\varrho, x) \in L^2_{x, loc}$ and $u_0 \in L^2_x$. Thanks to Lemma \ref{lem:conv_mollifier_with_weight}, this implies
$$
\begin{aligned}
\int_{\Omega} |(u + \nabla p_h)^\eps(\varrho, x)|^2 \, \psi(x) \diff x - \int_{\Omega}|u_0^\eps(x)|^2\,\psi(x) \diff x \to \\
\to\int_{\Omega} |(u + \nabla p_h)(\varrho, x)|^2 \, \psi(x) \diff x - \int_{\Omega}|u_0(x)|^2\,\psi(x) \diff x.
\end{aligned}
$$

\noindent \underline{Terms $(\alpha + \theta \beta): D \phi$.} As $(\overline{u} + \nabla \overline{p_h})^{\varepsilon} \in L^{\infty}_t L^{\infty}_{x, loc}$ and time derivatives are not involved, convergence results with $\delta \to 0$ and $\tau \to 0$ are trivial. Therefore, we focus on convergence $\varepsilon \to 0$. We first write
\begin{multline*}
    \int_{-\eta}^\varrho\int_{\Omega}(\overline{\alpha} + \theta \overline{\beta}) (t, x) : D(\psi(x)(\overline{u} + \nabla \overline{p_h})^\eps(t,x))^\eps\diff x \diff t = \\ = \int_0^\varrho \int_{\Omega}(\alpha + \theta \beta)(t, x) : D(\psi(x)(u + \nabla p_h)^\eps(t, x))^\eps\diff x \diff t.
\end{multline*}
When $\theta>0$, we can simply use that $\alpha + \theta \beta \in L^{s'_{\max}}_{t,x}$ and $Du, u \in L^{s_{\max}}_{t,x}$ (the latter by K\"{o}rn's inequality) to obtain
\begin{multline*}
 \int_0^\varrho \int_{\Omega}(\alpha + \theta \beta)(t, x) : D(\psi(x)(u + \nabla p_h)^\eps(t, x))^\eps\diff x \diff t \to \\
 \to  \int_0^\varrho \int_{\Omega}(\alpha + \theta \beta)(t, x) : D(\psi(x)(u + \nabla p_h)(t, x))\diff x \diff t.
\end{multline*}
For the case $\theta = 0$ we apply Theorem \ref{thm:approx_theorem} to obtain
$$
D(\psi(x)(u + \nabla p_h)^\eps(t,x))^\eps \to D(\psi(x) (u(t,x) + \nabla p_h) \text{ modularly in } L^{s(t,x)}(\Omega_T).
$$
Then, since $\alpha\in L^{s'(t,x)}$, we may apply Theorem~\ref{thm:modular_l1} and conclude
\begin{align*}
    \int_{0}^\varrho\int_{\Omega}\alpha(t, x) : D(\psi(x)(u+\nabla p_h)^\eps(t,x))^\eps\diff x \diff t \to \int_{0}^\varrho\int_{\Omega}\alpha(t, x) : D(\psi(x)(u(t,x) + \nabla p_h))\diff x \diff t.
\end{align*}
\underline{Term $\overline{f}\, \phi$.} We have
\begin{align*}
    \int_{-\eta}^\varrho \int_{\Omega}\overline{f}\cdot ((\overline{u} + \nabla \overline{p_h})^\eps \psi)^\eps\,\diff x \diff t = \int_{0}^\varrho \int_{\Omega}f^\eps \cdot ((u + \nabla p_h)^\eps \psi)\,\diff x \diff t.
\end{align*}
Next, thanks to the assumption on $f$, we know that
$$
f^{\eps} \to f \textrm{ strongly in } L^1_tL^2_x.
$$
On the other hand, we also have the weak$^*$ convergence result
$$
(u + \nabla p_h)^\eps \psi \rightharpoonup^* (u + \nabla p_h)\psi \textrm{ weakly$^*$ in } L^{\infty}_t L^2_{\textrm{loc}}.
$$
%
Hence, we get
\begin{align*}
    \int_{0}^\varrho \int_{\Omega}f \cdot ((u + \nabla p_h)^\eps \psi)^\eps\,\diff x \diff t \to \int_{0}^\varrho \int_{\Omega}f \cdot(u + \nabla p_h)\,\psi\, \diff x \diff t.
\end{align*}

\noindent \underline{Terms $\overline{p_j^i} \, \DIV \phi$ for $i = 1,...,N$ and $j = 1,2$.} Due to incompressibility of $u$ and thanks to the fact that $p_h$ is harmonic, we can write
\begin{align*}
\int_{-\eta}^\varrho \int_\Omega \overline{p_j^i}(t,x)\,\DIV ((\overline{u} + \nabla \overline{p_h})^\eps \psi)^\eps\diff x \diff t &= \int_{0}^\varrho \int_{
\Omega} (p^i_j)^\eps (u + \nabla p_h)^\eps \cdot \nabla\psi\diff x \diff t.
\end{align*}
Next, we decompose the integration domain onto $\mathcal{B}^i_{3r/2}$ and $\Omega \setminus \mathcal{B}^i_{3r/2}$. When $\varepsilon < \frac{r}{4}$, we can write
\begin{multline*}
\int_{0}^\varrho \int_{
\Omega} (p^i_j)^\eps (u + \nabla p_h)^\eps \nabla\psi\diff x \diff t =
\int_{0}^\varrho \int_{
\Omega \cap \mathcal{B}^i_{3r/2}} (p^i_j)^\eps (u\,\mathds{1}_{\mathcal{B}^i_{2r}} + \nabla p_h)^\eps \nabla\psi\diff x \diff t + \\
+ \int_{0}^\varrho \int_{
\Omega \setminus \mathcal{B}^i_{3r/2}} (p^i_j\, \mathds{1}_{\Omega \setminus \mathcal{B}^i_{5r/4}})^\eps (u + \nabla p_h)^\eps \nabla\psi\diff x \diff t.
\end{multline*}
We discuss now the cases $j=1,2$ separately. For $j = 1$, we have $p_1^i \in L^{r_i'(t)}_{t,x}$ from Lemma \ref{lem:integrability_of_pi} and $u \in L^{r_i(t)}((0,T)\times \mathcal{B}^i_{2r})$ from \eqref{eq:regularity_estimates_u} (recall that $r_i\le R_i$), so we can pass to the limit in the first term. For the second one, we use harmonic regularity of $p^i_j$ outside $\mathcal{B}^i_{r}$. Namely, we have $u \in L^{\infty}_t L^2_x$ and $\nabla p_h \in L^{\infty}_{t} L^{\infty}_{x, loc}$ so that $(u + \nabla p_h) \in L^{\infty}(0,T;L^2(\supp\, \psi))$. As $p_1^i \in L^{s'_{\max}}(0,T;L^2(\Omega \setminus \mathcal{B}^i_{5r/4}))$, cf. Lemma \ref{lem:integrability_of_pi}, the convergence is clear.\\

\noindent For $j = 2$ the proof is similar. More precisely, for the first term we observe that since $p_2^i \in L^{R_i(t)/2}_{t,x}$, it is sufficient that $u \in L^{(R_i(t)/2)'}((0,T)\times \mathcal{B}^i_{3r/2})$. This is the case because
    $$
    (R_i(t)/2)' \leq R_i(t) \iff R_i(t) \geq 3.
    $$
Recalling the definition of $R_i$ in Lemma~\ref{lem:decomposition_Omega}, using the fact that $q_i(t)\ge s_{\min}\ge \frac{3d+2}{d+2}$, we see that
$$
R_i(t)\ge s_{\min} \left(1+\frac{2}{d}\right) \ge \frac{3d+2}{d+2} \left(\frac{d+2}{d}\right)=s_0=3+\frac{2}{d}>3.
$$
The second term is controlled in exactly the same way as for $j = 1$.

\noindent \underline{Terms $\overline{p_j} \, \DIV \phi$ for $j = 3,4$.} Similarly as above, it is sufficient to study $\int_{0}^\varrho \int_{
\Omega} p_j^\eps \cdot (u + \nabla p_h)^\eps \nabla\psi\diff x \diff t$. For $j=3$ the proof is obvious because $p_3 \in L^{1}_tL^2_x$ so it is sufficient that $(u +\nabla p_h) \in L^{\infty}_tL^2_{\textrm{loc}}$ which is of course the case. For $j=4$, we observe that $\theta > 0$ so that we can use additional regularity from \eqref{eq:regularity_estimates_not_u_2} which yields $u \in L^{s_{\max}}_{t,x}$ by K\"{o}rn's and Poincar\'{e}'s inequalities. As $p_4 \in L^{s'_{\max}}_{t,x}$, the conclusion is clear.

\noindent \underline{Term  $\overline{u\otimes u} : \nabla \phi$.} We write as in the previous step
\begin{align*}
    \int_{-\eta}^\varrho \int_\Omega &\overline{u\otimes u} : \nabla (\psi (\overline{u} + \nabla \overline{p_h})^\eps)^\eps\diff x \diff t\\
    &= \int_0^\varrho\int_{\Omega} (u\otimes u)^\varepsilon : (\nabla\psi \otimes  (u + \nabla p_h)^\varepsilon)\diff x\diff t + \int_0^\varrho\int_\Omega (u\otimes u)^\varepsilon : (\psi\,(\nabla (u + \nabla p_h))^\varepsilon)\diff x\diff t.
\end{align*}
First, $u \otimes u \in L^{s_0/2}$ and $\nabla u$ and  $u$ belong to $L^{s_{\text{min}}}_{t,x}$. As $\nabla p_h \in L^{\infty}_t L^{\infty}_{x,loc}$, it is sufficient that $(s_0/2)' \leq s_{\text{min}}$. In fact, we have $s_{\text{min}} = (s_0/2)' $. Indeed, $s_0 = s_{\text{min}}\, \left(1+\frac{2}{d}\right)$ so that
$$
(s_0/2)' = \frac{s_0}{s_0 - 2} = \frac{s_{\text{min}} (d+2)}
{s_{\text{min}} (d+2) - 2d }
= \frac{3d+2}{3d+2 - 2d} = s_{\text{min}}
$$
and this concludes the proof.
\end{proof}

\section{The approximating problem}\label{section:5}
\noindent The crucial step of the existence  proof is the approximation of the stress tensor $S$. Namely, we set
\begin{align}\label{eq:def_of_reg_operator}
    S^\theta(t, x, \xi) := S(t, x, \xi) + \theta\nabla_{\xi}m(|\xi|), \qquad  m(|\xi|) := |\xi|^{s_{\max}}.
\end{align}
The advantage of such an approximation lies in the fact that function $S^\theta$ satisfies Assumption \ref{ass:stress_tensor} with $s(t,x)\equiv s_{\max}$, see Lemma \ref{lem:prop_of_Stheta} below. In particular, the analysis of problems with function $S^\theta$ is substantially easier and can be performed in usual Lebesgue spaces.\\

\noindent Now, we formulate the result concerning existence of solutions to the approximation problem
\begin{equation}\label{eq:approx_problem}
\begin{split}
\p_t(u^{\theta} + \nabla p_h^\theta) + \DIV (u^{\theta} \otimes u^{\theta}) &= \DIV S^{\theta}(t,x, Du^{\theta}) + f  + \sum_{i = 1}^N \nabla(p^{i,\theta}_1 + p^{i,\theta}_2) + \nabla(p_3 + p^{\theta}_4),\\
\DIV u^{\theta}&=0.
\end{split}
\end{equation}

\begin{thm}\label{thm:existence_approximation}
Let $S$ satisfy Assumption \ref{ass:stress_tensor} and $S^\theta$ be defined in \eqref{eq:def_of_reg_operator}. Then, for any $f\in L^{1}_t L^2_x$ and any initial condition $u_0\in L^2_{0,\DIV}(\Omega)$, there exists a function $u^\theta\in L^\infty_t L^2_x \cap L^{s_{\max}}_t W^{1,{s_{\max}}}_{0, x}$, $Du^\theta\in L^{s_{\max}}_{t,x}$, such that $S^{\theta}(t,x,Du^{\theta}) \in L^{s'_{\max}}_{t,x}$ and
\begin{align}\label{app:existence}
    \int_{\Omega_T}-u^\theta \cdot \p_t \phi - u^\theta\otimes u^\theta : \nabla\phi + S^\theta(t, x, Du^\theta):D\phi\diff x\diff t = \int_{\Omega_T}f \cdot \phi\diff x\diff t + \int_{\Omega}u_0(x) \cdot\phi(0,x)\diff x
\end{align}
for any vector-valued $\phi\in C^\infty_c([0, T) \times \Omega)$ fulfilling $\DIV \phi = 0$. Moreover, the following global energy equality is satisfied for all $t\in (0, T)$
\begin{align}\label{app:energy}
    \frac{1}{2}\int_{\Omega}|u^\theta(t, x)|^2\diff x - \frac{1}{2}\int_{\Omega}|u_0(x)|^2\diff x + \int_0^t\int_{\Omega}S^\theta(\tau, x, Du^\theta):Du^\theta\diff x\diff \tau = \int_0^t \int_{\Omega}f \cdot u^\theta\diff x \diff \tau.
\end{align}
\end{thm}
In the construction of a solution to \eqref{eq:the_main_result}, we want to let $\theta \to 0$ in \eqref{eq:approx_problem} and \eqref{app:existence}. To this end, we need certain estimates independent of $\theta$, which is the content of the next result.
\begin{thm}\label{res:estimate_on_approx_seq}
Let $\{u^{\theta}\}_{\theta \in (0,1)}$ be the sequence of solutions to \eqref{eq:approx_problem} constructed in Theorem \ref{thm:existence_approximation}. Let $\{p^{i,\theta}_1\}_{\theta \in (0,1)}$, $\{p^{i,\theta}_2\}_{\theta \in (0,1)}$, $\{p_3\}_{\theta \in (0,1)}$, $\{p^{\theta}_4\}_{\theta \in (0,1)}$, $\{p^{\theta}_h\}_{\theta \in (0,1)}$ be the sequences of pressures obtained by Theorem \ref{thm:local_energy_equality} with $
\alpha = S(t,x,Du^{\theta})$, $\beta = \nabla_{\xi}m(|D u^{\theta}|)
$. Then,
\begin{enumerate}[label=(B\arabic*)]
    \item\label{lemmathet_est1} $\{u^{\theta}\}_{\theta \in (0,1)}$ is bounded in $L^{\infty}_t L^2_x$,
    \item\label{lemmathet_est2} $\{Du^{\theta}\}_{\theta \in (0,1)}$ is bounded in $L^{s(t,x)}_{t,x}$,
   \item\label{lemmathet_est3i1/2} $\{ u^{\theta}\}_{\theta \in (0,1)}$ is bounded in $L^{s_{\min}}_t W^{1,s_{\min}}_{x,0}$ and $L^{q_i(t)}(0,T; W^{1,q_i(t)}(\mathcal{B}^i_{2r}))$,
    \item\label{lemmathet_est2and1/2} $\{u^{\theta}\}_{\theta \in (0,1)}$ is bounded in $L^{s_0}_{t,x}$ and $L^{R_i(t)}((0,T)\times \mathcal{B}^i_{2r})$,
    \item\label{lemmathet_est3}  $\{S(t,x,Du^{\theta})\}_{\theta \in (0,1)}$ is bounded in $L^{s'(t,x)}_{t,x}$,
    \item\label{lemmathet_est4} $\{\theta \, |D u^{\theta}|^{s_{\max}}\}_{\theta \in (0,1)}$ is bounded in $L^1_{t,x}$,
    \item\label{lemmathet_est5} $\{\theta^{1-s'_{\max}} \left|\theta\,\nabla_{\xi}m(|D u^{\theta}|)\right|^{s'_{\max}}\}_{\theta \in (0,1)}$ is bounded in $L^1_{t,x}$,
\item\label{lemmathet_est4and1/2-A} $\{p^{i,\theta}_1\}_{\theta\in(0, 1)}$ is bounded in $L^{r_i'(t)}_{t,x}$ and $L^{s'_{\max}}(0,T; L^{\infty}_{\text{loc}}(\R^d\setminus \mathcal{B}^i_{r}))$,
\item\label{lemmathet_est4and1/2-C} $\{p^{i,\theta}_2\}_{\theta\in(0, 1)}$ is bounded in $L^{R_i(t)/2}_{t,x}$ and $L^{s_0/2}(0,T; L^{\infty}_{\text{loc}}(\R^d\setminus \mathcal{B}^i_{r}))$,
\item\label{lemmathet_est4and1/2-F} $\{ \theta^{-1/s_{\max}} \, p^\theta_4\}_{\theta\in(0, 1)}$ is bounded in $L^{s'_{\max}}_{t,x}$
    \item\label{lemmathet_est9and5/8} $\{p^\theta_h\}_{\theta\in(0, 1)}$ is bounded in $L^\infty_t L^{s'_{\max}}_x$
    \item\label{lemmathet_est6} $\{p^{\theta}_h\}_{\theta\in (0, 1)}$ is bounded in $L^{\infty}_tW^{2,\infty}_{x, loc}$,
    \item\label{lemmathet_est8} $\{\partial_t u^\theta\}_{\theta\in (0, 1)}$ is bounded in $L^1_t V_{2,d}^*$,
    \item\label{lemmathet_est7} $\{ \partial_t (u^{\theta} + \nabla p_h^\theta) \}_{\theta \in (0,1)}$ is bounded in $L^{1}_t\left( W^{1,s_{\max}}_{x,0}\right)^*$.
\end{enumerate}
where $V_{2,d}$ is closure of $\{\phi\in C^\infty_c(\Omega) | \DIV\phi = 0\}$ in $W^{2,d}(\Omega)$.
\end{thm}
The rest of this section is devoted to the proofs of Theorems \ref{thm:existence_approximation} and \ref{res:estimate_on_approx_seq}. We begin by establishing certain properties of function $S^\theta$.
\begin{lem}\label{lem:prop_of_Stheta}
Function $S^\theta$ satisfies the following:
\begin{enumerate}[label=(R\arabic*)]
\item \label{monotonicity_theta} $S^{\theta}(t,x, \xi)$ is a Carath\'{e}odory function and $S(t, x, 0) = 0$,
\item \label{coercitivity_stress_tensor_theta} (coercitivity and growth in $L^{s(t,x)}$) there exists a positive constant $c$ and a non-negative, integrable function $h(t, x)$, such that for any $\xi\in\symm$ and almost every $(t, x) \in (0, T) \times \Omega$
$$
c \, S^{\theta}(t, x, \xi) : \xi \geq | \xi|^{s(t, x)} + | S(t, x, \xi)|^{s'(t, x)} + \theta\nabla_{\xi}m(|\xi|) \cdot \xi - h(t, x);
$$
the constant $c$ and function $h$ can be chosen independently of $\theta$,
\item \label{coercitivity_stress_tensor_theta_r} (coercitivity and growth in $L^{s_{\max}}$) there exists a positive constant $c^{\theta}$ and a non-negative, integrable function $h^{\theta}(t, x)$, such that for any $\xi\in\symm$ and almost every $(t, x) \in (0, T) \times \Omega$
$$
c^{\theta} \, S^{\theta}(t, x, \xi) : \xi \geq | \xi|^{s_{\max}} + | S^{\theta}(t, x, \xi)|^{s'_{\max}} - h^\theta(t, x)
$$
\item \label{monotonicity_stress_tensor_theta}(monotonicity) S is strictily monotone, i. e.:
$$
(S^{\theta}(t, x, \xi_1) - S^{\theta}(t, x, \xi_2)) : (\xi_1 - \xi_2) > 0
$$
for all $\xi_1 \neq \xi_2\in \symm$ and almost every $(t, x) \in (0, T) \times \Omega$.
\end{enumerate}
\end{lem}
\begin{proof}
Properties \ref{monotonicity_theta} and \ref{monotonicity_stress_tensor_theta} are fairly obvious. To see \ref{coercitivity_stress_tensor_theta} and \ref{coercitivity_stress_tensor_theta_r}, we first note that by the definition of the convex conjugate we have
\begin{align}\label{eq:gradient_Young_equality}
    \nabla_\xi m(|\xi|)\cdot\xi = |\xi|^{s_{\max}}+C_*\,|\nabla_\xi m(|\xi|)|^{s'_{\max}}, \qquad C_* := \frac{1}{s'_{\max}\,s_{\max}^{s'_{\max}-1}}.
\end{align}
As $S$ satisfies \ref{coercitivity_stress_tensor} in Assumption \ref{ass:stress_tensor}, we have
\begin{align*}
    c\, S^\theta(t, x, \xi) \cdot \xi &\geq |\xi|^{s(t, x)} + |S(t, x, \xi)|^{s'(t, x)} - h(t, x) + c\,\theta \nabla_{\xi}m(|\xi|)\cdot\xi
\end{align*}
so that we obtain \ref{coercitivity_stress_tensor_theta}. To see \ref{coercitivity_stress_tensor_theta_r}, we estimate term $S^\theta(t, x, \xi) \cdot \xi$ more carefully using \eqref{eq:gradient_Young_equality}:
\begin{align*}
    c\, S^\theta(t, x, \xi) \cdot \xi &\geq |\xi|^{s(t, x)} + |S(t, x, \xi)|^{s'(t, x)} - h(t, x) + c\,\theta \nabla_{\xi}m(|\xi|)\cdot\xi\\
    &\geq 0 + |S(t, x, \xi)|^{s'_{\max}} - 1 - h(t, x) + c\,\theta \,|\xi|^{s_{\max}} + c\,\theta \,C_*\,|\nabla_{\xi}m(|\xi|)|^{s'_{\max}},
\end{align*}
where we estimated $|S(t, x, \xi)|^{s'(t, x)} \geq |S(t, x, \xi)|^{s'_{\max}} - 1$ which follows from $s'(t,x) \geq s'_{\max}$. Applying Jensen's inequality, we obtain
\begin{align*}
    |S(t, x, \xi)|^{s'_{\max}} &- h(t, x) + c\,\theta \,|\xi|^{s_{\max}} + c\,\theta \, C_* \,|\nabla_{\xi}m(|\xi|)|^{s'_{\max}} \geq \\
    & \geq 2\min{(1, c\,C_*)}\left(\frac{1}{2}|S(t, x, \xi)|^{s'_{\max}} + \frac{1}{2}|\theta \nabla_{\xi}m(|\xi|)|^{s'_{\max}}\right) + c\,\theta|\xi|^{s_{\max}} - h(t,x) - 1\\
    &\geq 2\min{(1, c\,C_*)}\left|\frac{1}{2}S^\theta(t, x, \xi)\right|^{s'_{\max}} + c\,\theta|\xi|^{s_{\max}} - h(t, x) - 1\\
    &\geq \min{(\min{(1, c\,C_*)} \, 2^{1-s'_{\max}}, c\,\theta)}(|S^\theta(t, x, \xi)|^{s'_{\max}} + |\xi|^{s_{\max}}) - h(t, x) - 1
\end{align*}
Thus, taking
$$
c^{\theta} := \frac{c}{\min{(\min{(1, c\,C)}2^{1-s'_{\max}}, c\,\theta)}}\text{, } \qquad h^{\theta}(t, x) := \frac{h(t, x) + 1}{\min{(\min{(1, c\,C)}2^{1-s'_{\max}}, c\,\theta)}}
$$
concludes the proof of \ref{coercitivity_stress_tensor_theta_r}.
\end{proof}

\begin{proof}[Proof of Theorem \ref{thm:existence_approximation}] The proof of this theorem follows the lines of the proof in \cite{gwiazda2008onnonnewtonian}. The only difference is the dependence of stress tensor $S$ on time variable.
\end{proof}

\begin{proof}[Proof of Theorem \ref{res:estimate_on_approx_seq}]
We combine the energy equality~\eqref{app:energy} and coercivity estimate~\ref{coercitivity_stress_tensor_theta} in~Lemma~\ref{lem:prop_of_Stheta} to deduce
\begin{multline*}
   \frac{1}{2}\int_{\Omega}|u^\theta(t, x)|^2\diff x + \frac{1}{c} \int_{\Omega_t} \left(
   | Du|^{s(\tau, x)}
   +  | S(\tau, x, Du)|^{s'(\tau, x)}
   + c\,\theta\nabla_{\xi}m(|D u^{\theta}|) \cdot D u^{\theta} \right) \diff x \diff \tau = \\ =  \frac{1}{2}\int_{\Omega}|u_0(x)|^2\diff x  + \int_0^t \int_{\Omega}f\cdot u^\theta\diff x \diff \tau +  \int_0^t \int_{\Omega} h(\tau,x) \diff x \diff \tau.
 \end{multline*}
Using the H\"{o}lder inequality, we can estimate $\int_{\Omega} f\cdot u^{\theta}\le \|f\|_2(\|u^{\theta}\|_2^2+1)$ on the right hand side.  Using the Gr\"{o}nwall lemma and also the assumptions on $f$, $u_0$ and $h$, we deduce the right hand side is bounded independently of $\theta$ and consequently, we conclude the proof of \ref{lemmathet_est1}, \ref{lemmathet_est2} and \ref{lemmathet_est3}. Moreover, it shows that $\int_{\Omega_T} \theta\nabla_{\xi}m(|D u^{\theta}|) \cdot D u^{\theta}  \diff x \diff \tau$ is bounded uniformly in $\theta \in (0,1)$. But then, using \eqref{eq:gradient_Young_equality} we deduce
 $$
 \int_{\Omega_T} \left(\theta \, |Du|^{s_{\max}} + \theta \, C_r\,|\nabla_\xi m(|Du|)|^{s'_{\max}}\right) \diff x \diff \tau = \int_{\Omega_T} \theta\nabla_{\xi}m(|D u^{\theta}|) \cdot D u^{\theta}  \diff x \diff \tau.
 $$
 This implies \ref{lemmathet_est4} and \ref{lemmathet_est5}.\\

\noindent To see \ref{lemmathet_est3i1/2}, we observe that \ref{lemmathet_est2} implies that $\{Du^{\theta}\}_{\theta \in (0,1)}$ is bounded in $L^{s_{\min}}_{t,x}$ and $L^{q_i(t)}_{t,x}((0,T)\times \mathcal{B}^i_{2r})$ (because $s_{\min} \leq s(t,x)$ on $\Omega_T$ and $q_i(t) \leq s(t,x)$ on $(0,T)\times \mathcal{B}^i_{2r}$). Then, K\"{o}rn's inequality implies that $\{\nabla u^{\theta}\}_{\theta \in (0,1)}$ is bounded in $L^{s_{\min}}_{t,x}$ and $L^{q_i(t)}_{t,x}((0,T)\times \mathcal{B}^i_{2r})$. To conclude the estimate, we note that the $\int_{\Omega}u^{\theta}(t,x)\diff x$ is controlled in $L^{\infty}_t$ by \ref{lemmathet_est1} so that the claim follows by the Poincar\'{e} inequality. Next, estimate \ref{lemmathet_est2and1/2} follows from \ref{lemmathet_est1} and \ref{lemmathet_est3i1/2} together with Lemma~\ref{thm:interpolation}.\\

\noindent To obtain estimates on the pressures we apply Theorem \ref{thm:local_energy_equality} and Lemma \ref{lem:integrability_of_pi} with
$$
\alpha = S^{\theta}(t,x,Du^{\theta}), \qquad \beta = \nabla_{\xi}m(|D u^{\theta}|)
$$
so that \ref{lemmathet_est4and1/2-A}--\ref{lemmathet_est6} follows from \ref{lemmathet_est2and1/2}, \ref{lemmathet_est3} and \ref{lemmathet_est5}.\\

\noindent The bound \ref{lemmathet_est8} can be obtained by the following argument. We have (for divergence-free distributional formulation):
$$
\p_t u^\theta = \underbrace{- \DIV (u^{\theta} \otimes u^{\theta}) + \DIV S^{\theta}(t,x, Du^{\theta})  + f}_{:=\, A^{\theta}}
$$
We want to prove that $A^\theta$ defines a functional on $L^\infty_t V_{2,d}$. This is clear because functions in $L^\infty_t V_{2,d}$ have spatial derivatives in $L^{\infty}_t L^{z}_x$ for all $z<\infty$ and all the functions $u^{\theta} \otimes u^{\theta}$, $S^{\theta}(t,x, Du^{\theta})$ and $f$ are at least uniformly integrable in time and belong to some $L^a$ with $a>1$ with the norm independent of $\theta \in (0,1)$.\\

\noindent Now, we move to establishing the regularity of the time derivative $\partial_t (u^\theta + \nabla p_h^\theta)$ as in \ref{lemmathet_est7}. In view of \eqref{eq:def_of_reg_operator} and \eqref{eq:approx_problem}, we can write (in the sense of distributions)
$$
\p_t(u^{\theta} + \nabla p_h^\theta) ={- \DIV (u^{\theta} \otimes u^{\theta}) + \DIV S^{\theta}(t,x, Du^{\theta})  + f} + \,\sum_{i=1}^N\nabla (p^{i,\theta}_1 + p^{i,\theta}_2) + \nabla (p_3 + p^{\theta}_4).
$$
We observe that all of the functions $u^{\theta} \otimes u^{\theta}$, $S^{\theta}(t,x, Du^{\theta})$, $f$, $p_i^{\theta}$ are uniformly bounded at least in $L^1_tL^{s'_{\max}}_{x}$ (this uses inequalities $s'_{\max} \leq s'_{\min}$ and equality $s_0/2 = s'_{\min}$) so that $\partial_t(u^{\theta}+ \nabla p_h^{\theta})$ is bounded in $L^{1}_t (W^{1,s'_{\max}}_{x_0})^*$, hence \ref{lemmathet_est7} holds. 
\end{proof}

\section{Proof of existence result via the monotonicity method}\label{section:6}
\begin{proof}[Proof of Theorem \ref{thm:the_main_result}] The proof is divided into four steps.\\

\noindent \underline{Step 1: Approximating problem and compactness.}
Let $u^{\theta}$ be a solution to \eqref{eq:approx_problem} constructed in Theorem \ref{thm:existence_approximation}. Let $p^{i,\theta}_1$, $p^{i,\theta}_2$, $p_3$, $p^{\theta}_4$, $p^\theta_h$ be the sequences of pressures obtained in Theorem \ref{thm:local_energy_equality}. First, thanks to Theorem \ref{res:estimate_on_approx_seq}, we can extract appropriate subsequences such that
\begin{enumerate}[label=(C\arabic*)]                        
\item\label{c0} $u^{\theta}\overset{*}{\rightharpoonup} u$ in $L^{\infty}_{t}L^2_x$,
\item\label{conv_item_strongu} $u^{\theta} \to u$ a.e. in $\Omega_T$ and in $L^{c}_{t,x}$ for all $c < s_0$,
    \item\label{conv_item_strongu_c1c2} $u^{\theta} \to u$ in $L^{c_1}_{t} L^{c_2}_{x}$ for all $c_1 < \infty$ and $c_2<2$,
    \item\label{conv_item_strongu_local} $u^{\theta} \to u$ in $L^{R_i(t) - \delta}((0,T)\times \mathcal{B}^i_{2r})$ for all $\delta>0$,
    \
    \item\label{conv_item_sing_u} $\theta^{1/s_{\max}}\, u^{\theta} \to 0$ in $L^{s_{\max}}_{t,x}$,
    \item\label{conv_item_weak_monotone_operator_S} $S(t,x,Du^{\theta}) \rightharpoonup \chi$ in $L^{s'(t,x)}(\Omega_T)$ and $L^{r_i'(t)}_{t,x}((0,T)\times \mathcal{B}^i_{2r})$ for some $\chi \in L^{s'(t,x)}(\Omega_T)$,
    \item\label{conv_item_Dutheta} $Du^{\theta} \rightharpoonup Du$ weakly in $L^{s(t,x)}(\Omega_T)$,
    \item\label{conv_item_L1_weak} $\theta\,\nabla_{\xi}m(|D u^{\theta}|) \to 0$ in $L^{s_{\max}}_{t,x}$,
\item\label{conv_item_est4and1/2-A} $p^{i,\theta}_1 \overset{*}{\rightharpoonup} \widetilde{p^i_1}$ in $L^{r_i'(t)}_{t,x}$ and $L^{s'_{\max}}(0,T; L^{\infty}_{\text{loc}}(\R^d\setminus \mathcal{B}^i_{r}))$,
\item\label{conv_item_est4and1/2-C} $p^{i,\theta}_2 \overset{*}{\rightharpoonup} \widetilde{p^i_2}$ in $L^{R_i(t)/2}_{t,x}$ and $L^{s_0/2}(0,T; L^{\infty}_{\text{loc}}(\R^d\setminus \mathcal{B}^i_{r}))$,
\item\label{conv_item_est4and1/2-F} $p^\theta_4 \to 0$ in $L^{s'_{\max}}_{t,x}$,
    \item\label{conv_item_est9and5/8} $p^\theta_h \overset{*}{\rightharpoonup} \widetilde{p_h}$ in $L^{s'_{\max}}_{t,x}$ and $L^{\infty}_tW^{2,\infty}_{x, loc}$,
    \item\label{conv_item_est_add_1} $\nabla p^\theta_h \to \nabla \widetilde{p_h}$ in $L^{c}_{t} L^c_{x,loc}$ for all $c <\infty$,
    \item\label{conv_item_est_add_2/5} $\nabla^2 p^\theta_h \to \nabla^2 \widetilde{p_h}$ in $L^{c}_{t} L^c_{x,loc}$ for all $c <\infty$,
    \item\label{conv_item_est_add_2} $D(\nabla p^\theta_h) \to D(\nabla \widetilde{p_h})$ in $L^{c}_{t} L^c_{x,loc}$ for all $c <\infty$,
    \item\label{conv_item_spaceintegrals} $\int_{\Omega} |(u^{\theta} + \nabla p_h^\theta)(t,x)|^2\psi(x) \diff x \to \int_{\Omega} |(u + \nabla \widetilde{p_h})(t,x)|^2\psi(x) \diff x$ for a.e. $t \in [0,T]$ and $\psi \in C^\infty_c(\Omega)$,
\end{enumerate}
Indeed, the strong convergence in $L^{c}_{t,x}$ in \ref{conv_item_strongu} follows by interpolation: the sequence $\{u^{\theta}\}$ is bounded in $L^{s_0}_x$ (see \ref{lemmathet_est2and1/2}) and $\{u^{\theta}\}$ is strongly compact in $L^1_{t,x}$ by Aubin--Lions lemma \ref{aubin-lions} (this uses \ref{lemmathet_est3i1/2} and \ref{lemmathet_est8}). Similarly, we obtain \ref{conv_item_strongu_c1c2} and \ref{conv_item_strongu_local}, this time exploiting uniform bounds of $\{u^{\theta}\}$ in $L^2_t L^{\infty}_x$ and $L^{R_i(t)}((0,T)\times \mathcal{B}^i_{2r})$. To see \ref{conv_item_sing_u}, we note that \ref{lemmathet_est4} and K\"{o}rn's inequality implies uniform bound $\{\theta^{1/s_{\max}} \nabla u^{\theta}\}$ in $L^{s_{\max}}_{t,x}$ so that by Sobolev embedding and Dirichlet boundary condition we have uniform bound $\{\theta^{1/s_{\max}} u^{\theta}\}$ in $L^{c}_{t,x}$ for some $c > s_{\max}$. As $\theta^{1/s_{\max}} u^{\theta} \to 0$ in $L^1_{t,x}$, we conclude by interpolation. Next, convergence results \ref{conv_item_weak_monotone_operator_S}--\ref{conv_item_est9and5/8}
follow from Banach--Alaoglu theorem and estimates \ref{lemmathet_est3}, \ref{lemmathet_est2} and \ref{lemmathet_est5}--\ref{lemmathet_est6},
respectively. Next, we can use \ref{lemmathet_est3i1/2}, \ref{lemmathet_est6},  \ref{lemmathet_est7} and the Aubin--Lions \ref{aubin-lions} lemma to conclude that 
$$
(u^{\theta}+\nabla p_h^{\theta}) \to (u+\nabla \widetilde{p_h}) \qquad \textrm{ in } L^{1}_{t} L^{1}_{x, loc}.
$$
Thus, \ref{conv_item_est_add_1} follows from \ref{lemmathet_est6}. Finally, since $p_h^{\theta}$ is harmonic with respect to the spatial variable, we have that $\|p^{\theta}_h - \widetilde{p_h}\|_{W^{k,2}(\Omega'')}\le C(k,\Omega'', \Omega')\|p^{\theta}_h - \widetilde{p_h}\|_{L^1(\Omega')}$ for all $\Omega'' \Subset \Omega' \subset \Omega$ and all $k$. Consequently, \ref{conv_item_est_add_2/5} and \ref{conv_item_est_add_2} follow from \ref{conv_item_est_add_1}. The last property~\ref{conv_item_spaceintegrals} holds true because of the presence of the function $\psi$ having compact support in $\Omega$ and thus we can combine \ref{conv_item_est_add_1} and \ref{conv_item_strongu_c1c2} and use the classical properties of the Lebesgue spaces.\\

\noindent Now, for each $\theta \in (0,1)$ we use Theorem~\ref{thm:existence_approximation} to have a distributional formulation without pressure:
\begin{align}\label{app:existence_in_final_proof}
    \int_{\Omega_T}-u^\theta \cdot \p_t\phi - u^\theta\otimes u^\theta : \nabla\phi + S^\theta(t, x, Du^\theta):D\phi\diff x\diff t = \int_{\Omega_T}f \cdot \phi\diff x\diff t + \int_{\Omega}u_0(x)\cdot \phi(0,x)\diff x.
\end{align}
satisfied for all vector-valued  $\phi\in C^\infty_c([0, T) \times \Omega)$ fulfilling  $\DIV \phi = 0$.  We can let $\theta \to 0$ in \eqref{app:existence_in_final_proof} to obtain
\begin{align}\label{app:existence_in_final_proof_after_limit}
    \int_{\Omega_T}-u \cdot \p_t \phi- u \otimes u : \nabla\phi + \chi:D\phi\diff x\diff t = \int_{\Omega_T}f \cdot\phi\diff x\diff t + \int_{\Omega}u_0(x) \cdot\phi(0,x)\diff x.
\end{align}
The only nontrivial step in the passage to the limit above concerns operator $S^{\theta}(t,x,Du^{\theta})$. However, by \eqref{eq:def_of_reg_operator}, we may write $S^{\theta}(t,x,Du^{\theta}) = S(t,x,Du^{\theta}) + \theta \nabla_{\xi} m(|Du^{\theta}|)$. Then, by \ref{conv_item_L1_weak}, we know that the regularizing term converges in $L^1_t L^1_x$ which is sufficient to perform the desired passage to the limit $\theta \to 0$.\\

\noindent In view of \eqref{app:existence_in_final_proof_after_limit}, the proof of existence of solutions will be concluded if we prove $\chi(t,x) = S(t,x,Du)$.\\

\noindent \underline{Step 2: Local energy equalities.} Applying Theorem~\ref{thm:local_energy_equality}, we also have a distributional formulation with pressure
\begin{equation}\label{eq:weak_form_with_press_and_IC_in_final_proof}
\begin{split}
    \int_{\Omega_T}-&(u^{\theta} + \nabla p_h^\theta)\cdot \p_t\phi - u^{\theta}\otimes u^{\theta}:\nabla\phi \,+\, S^{\theta}(t,x,Du^{\theta}) : D\phi\diff x \diff t = \\ 
    = \int_{\Omega}&u_0(x)\cdot \phi(0, x) \diff x -  \int_{\Omega_T} \left(\sum_{i = 1}^N(p_1^{i,\theta}+p_2^{i,\theta} ) + p_3 + p_4^{\theta} \right)\,\DIV\phi\diff x \diff t + \int_{\Omega_T}f \cdot \phi\diff x \diff t
\end{split}
\end{equation}
satisfied for all $\phi\in C^\infty_c([0, T) \times \Omega)$. We can let $\theta \to 0$ in \eqref{eq:weak_form_with_press_and_IC_in_final_proof} similarly as above to obtain
\begin{equation}\label{eq:weak_form_with_press_and_IC_in_final_proof_after_limit}
\begin{split}
    \int_{\Omega_T}-&(u + \nabla \widetilde{p_h}) \cdot \p_t\phi - u\otimes u:\nabla\phi \,+\, \chi : D\phi\diff x \diff t = \\
     = \int_{\Omega}&u_0(x) \cdot\phi(0, x) \diff x - \left(\int_{\Omega_T}\sum_{i = 1}^N (\widetilde{p^i_1} +\widetilde{p^i_2}) + \widetilde{p_3} + \widetilde{p_4}\right) \,\DIV\phi\diff x \diff t + \int_{\Omega_T}f \cdot \phi\diff x \diff t.
\end{split}
\end{equation}
On the other hand, we may apply Theorem~\ref{thm:local_energy_equality} directly to \eqref{app:existence_in_final_proof_after_limit}. This yields pressures $p^i_1$, $p^i_2$, $p_3$, $p_4$ and $p_h$ with a distributional formulation as \eqref{eq:weak_form_with_press_and_IC_in_final_proof_after_limit} but with $\widetilde{p_j^i}$, $\widetilde{p_j}$ and $\widetilde{p_h}$ replaced by ${p_j^i}$, $p_j$ and $p_h$ respectively. By the uniqueness (linearity) in the Lemma \ref{lem:integrability_of_pi}, we obtain $\widetilde{p_j^i} = p_j^i$ and $\widetilde{p_j} = p_j$ almost everywhere. On the other hand, $p_h$ is obtained from the Ne\v{c}as theorem \ref{lem:deRham} uniquely up to the condition
$$
\int_{\Omega}p_h(t, x)\diff x = 0.
$$
But from the weak convergence \ref{conv_item_est9and5/8}, the strong convergence \ref{conv_item_est_add_1} and the Poincar\'{e} inequality, we may deduce that for almost all $t\in (0,T)$
$$
0 = \lim_{\theta\to 0}\int_{\Omega}p_h^\theta(t, x)\diff x = \int_{\Omega}\widetilde{p_h}(t, x)\diff x.
$$
Hence $\widetilde{p_h} = p_h$.
For further reference, we recall local energy equalities obtained from \eqref{app:existence_in_final_proof} and \eqref{app:existence_in_final_proof_after_limit} by Theorem \ref{thm:local_energy_equality}. There hold
\begin{equation}\label{local_energy_equality_theta}
\begin{split}
    \frac{1}{2}&\int_\Omega |u^{\theta}(t, x) + \nabla p_h^{\theta}(t, x)|^2 \, \psi(x) \diff x + \int_{0}^t \int_\Omega  S^{\theta}(\tau, x, Du^\theta):D(\psi(x)(u^{\theta}+\nabla p_h^\theta)(\tau,x)) \diff x \diff \tau = \\
    &= \frac{1}{2}\int_\Omega  |u_0(x)|^2 \, \psi(x) \diff x + \int_{0}^t \int_\Omega (u^{\theta}\otimes u^{\theta}) : \nabla (\psi (u^{\theta} + \nabla p_h^{\theta}))\diff x\diff \tau + \\
    &+  \int_{0}^t\int_{\Omega}f\cdot(u^{\theta} + \nabla p_h^{\theta})\psi \diff x \diff \tau - \int_{0}^t \int_{\Omega} \left(\sum_{i=1}^N (p_1^{i,\theta} + p_2^{i,\theta}) + p_3 + p_4^{\theta} \right)(u^{\theta} + \nabla p_h^{\theta} )\cdot\nabla\psi \diff x \diff \tau
\end{split}
\end{equation}
and also 
\begin{equation}\label{local_energy_equality_after_limit_pass}
\begin{split}
    \frac{1}{2}&\int_\Omega |u(t, x) + \nabla p_h(t, x)|^2 \, \psi(x) \diff x + \int_{0}^t \int_\Omega  \chi(\tau, x):D(\psi(x)(u + \nabla p_h)(\tau,x)) \diff x \diff \tau = \\
    &= \frac{1}{2}\int_{\Omega}  |u_0(x)|^2 \, \psi(x) \diff x + \int_{0}^t \int_\Omega  (u\otimes u) : \nabla (\psi (u + \nabla p_h))\diff x\diff \tau + \\
    &+  \int_{0}^t\int_{\Omega}f\cdot(u + \nabla p_h)\psi\diff x \diff \tau -  \int_{0}^t \int_{\Omega} \left(\sum_{i=1}^N (p_1^i + p_2^i) + p_3  \right)(u + \nabla p_h )\cdot \nabla\psi \diff x \diff \tau,
\end{split}
\end{equation}
for a.e. $t \in (0,T)$ and all test functions $\psi \in C_c^{\infty}(\Omega)$. The idea is to compare \eqref{local_energy_equality_theta} with \eqref{local_energy_equality_after_limit_pass} in the limit $\theta \to 0$ to identify $\chi$ via monotonicity arguments.
\\

\noindent \underline{Step 3: Limits of the pressure terms $p_j^{i,\theta}$.} In this step, we prove for $i=1,...,N$ and $j = 1, 2$ and  that
\begin{equation}\label{eq:conv_press_p_1_p_2_theta}
\int_{0}^t \int_{\Omega} p_j^{i,\theta}(u^{\theta} + \nabla p_h^{\theta} )\cdot\nabla\psi \diff x \diff \tau\to \int_{0}^t \int_{\Omega} p_j^i(u + \nabla p_h )\cdot\nabla\psi \diff x \diff \tau
\end{equation}

\noindent Let $j=1$ and $i \in \{1,...,N\}$ be fixed. First, we split the integral for $\Omega \cap \mathcal{B}^{i}_{2r}$ and $\Omega \setminus \mathcal{B}^{i}_{2r}$. We treat the resulting terms separately.
\begin{itemize}
    \item On $\mathcal{B}^{i}_{2r}$ we have the weak convergence of $p_1^{i,\theta}$ in $L^{r_i'(t)}_{t,x}$ so it is sufficient to have strong convergence of $u^{\theta} + \nabla p_h^{\theta}$ in $L^{r_i(t)}_{t,x,loc}$ thanks to the compact support of $\psi$. This follows from \ref{conv_item_strongu_local} and \ref{conv_item_est_add_1} as $R_i(t)-r_i(t)\geq \frac{s_{\min}}{d}$.
    \item On $\Omega \setminus \mathcal{B}^{i}_{2r}$ we use the weak$^*$ convergence of $p_1^{i,\theta}$ in $L^{s'_{\max}}_t L^{\infty}_x$ from \ref{conv_item_est4and1/2-A} and local strong convergence of $u^{\theta} + \nabla p_h^{\theta}$ in $L^{s_{\max}}_t L^1_x$ from \ref{conv_item_strongu_c1c2} and \ref{conv_item_est_add_1}.
\end{itemize}

\noindent Now, let $j = 2$ and $i \in \{1,...,N\}$ be fixed. As above, we split the integral for $\Omega \cap \mathcal{B}^{i}_{2r}$ and $\Omega \setminus \mathcal{B}^{i}_{2r}$.
\begin{itemize}
    \item On $\mathcal{B}^{i}_{2r}$ we have the weak convergence of $p_2^{i,\theta}$ in $L^{R_i(t)/2}_{t,x}$ so it is sufficient to have strong convergence of $u^{\theta} + \nabla p_h^{\theta}$ in $L^{(R_i(t)/2)'}_{t,x,loc}$. However, we have $\left(\frac{R_i(t)}{2} \right)' < R_i(t)$  because $R_i(t) > 3$ (note that we already checked this in Section~\ref{section:4} below \eqref{eq:regularity_estimates_not_u_2}). Therefore, the required strong convergence follows from \ref{conv_item_strongu_local} and \ref{conv_item_est_add_1}.
    \item On $\Omega \setminus \mathcal{B}^{i}_{2r}$ we use the weak$^*$ convergence of $p_2^{i,\theta}$ in $L^{s_0/2}_t L^{\infty}_x$ from \ref{conv_item_est4and1/2-C} and local strong convergence of $u^{\theta} + \nabla p_h^{\theta}$ in $L^{s_0/2}_t L^1_x$ from \ref{conv_item_strongu_c1c2} and \ref{conv_item_est_add_1}.
\end{itemize}

\noindent \underline{Step 4: Limits of the other terms.} First, we notice that a direct application of \ref{conv_item_spaceintegrals}, $f \in L^{1}_t L^2_{x}$, $p_3 \in L^{1}_tL^2_x$ (Lemma \ref{lem:integrability_of_pi}), \ref{conv_item_est9and5/8}, \ref{conv_item_est_add_1}, \ref{c0} and  \ref{conv_item_strongu} yields for almost all $t\in (0,T)$
\begin{align}\label{eq:conv_energy_with_theta_1}
    \frac{1}{2}\int_\Omega |u^{\theta}(t, x) + \nabla p_h^{\theta}(t, x)|^2\psi(x)\diff x \to \frac{1}{2}\int_\Omega |u(t, x) + \nabla p_h(t, x)|^2\psi(x)\diff x,
\end{align}
\begin{equation}\label{eq:conv_energy_with_theta_2}
    \int_{0}^t \int_{\Omega} f\cdot (u^\theta + \nabla p_h^{\theta}) \psi \diff x \diff \tau \to \int_{0}^t \int_{\Omega}   f \cdot(u + \nabla p_h)\psi \diff x \diff \tau,
\end{equation}
\begin{equation}\label{eq:convergence_term_p_3_theta}
\int_{0}^t \int_{\Omega} p_3(u^{\theta} + \nabla p_h^{\theta} )\cdot\nabla\psi \diff x \diff \tau\to \int_{0}^t \int_{\Omega} p_3(u + \nabla p_h )\cdot\nabla\psi \diff x \diff \tau.
\end{equation}
Similarly, we also have
\begin{equation}\label{eq:convergence_term_p_4_theta}
\int_{0}^t \int_{\Omega} p_4^{\theta}(u^{\theta} + \nabla p_h^{\theta} )\cdot\nabla\psi \diff x \diff \tau\to 0
\end{equation}
because we can estimate
$$
\left|\int_{0}^t \int_{\Omega} p_4^{\theta}(u^{\theta} + \nabla p_h^{\theta} )\nabla\psi \diff x \diff \tau\right| \leq \|p_4^{\theta}\, \theta^{-1/s_{\max}}\|_{L^{s'_{\max}}_{t,x}} \, \|\theta^{1/s_{\max}}(u^{\theta} + \nabla p_h^{\theta} )\nabla\psi \|_{L^{s_{\max}}_{t,x}} \to 0
$$
due to estimate \ref{lemmathet_est4and1/2-F} and convergences \ref{conv_item_sing_u} and \ref{conv_item_est_add_1}. Now we want to prove that
\begin{equation}\label{eq:conv_energy_with_theta_3}
\int_{0}^t \int_\Omega (u^{\theta}\otimes u^{\theta}) : \nabla (\psi (u^{\theta} + \nabla p_h^\theta))\diff x\diff \tau \to \int_{0}^t \int_\Omega (u\otimes u) : \nabla (\psi (u + \nabla p_h))\diff x\diff \tau.
\end{equation}
We split $\psi (u^{\theta} + \nabla p_h^\theta) = \psi u^{\theta} + \psi \nabla p_h^\theta$. The convergence for the term $\nabla p_h^{\theta}$ is a simple consequence of \ref{conv_item_est_add_1}, \ref{conv_item_est_add_2/5} and $u^{\theta} \to u$ in $L^2_{t,x}$ from \ref{conv_item_strongu}. Therefore, we focus on $\int_{0}^t \int_\Omega (u^{\theta}\otimes u^{\theta}) : \nabla (\psi u^{\theta})\diff x\diff \tau$. We easily compute
$$
\int_{0}^t \int_\Omega (u^{\theta}\otimes u^{\theta}) : \nabla (\psi u^{\theta}) \diff x\diff \tau=  \frac{1}{2}
\int_{0}^t \int_\Omega |u^{\theta}|^2\, u^{\theta} \cdot \nabla \psi \diff x\diff \tau - \frac{1}{2}\int_{0}^t \int_\Omega \psi \DIV u^{\theta}\,  |u^{\theta}|^2 \diff x\diff \tau.
$$
The first term converges to $\frac{1}{2}
\int_{0}^t \int_\Omega |u|^2\, u \cdot \nabla \psi \diff x\diff \tau$ because $u^{\theta} \to u$ strongly in $L^3_{t,x}$ as in \ref{conv_item_strongu} (note that $s_0>3$). The second term vanishes by the incompressibility condition so that we obtain \eqref{eq:conv_energy_with_theta_3}. \\

\noindent Collecting \eqref{eq:conv_press_p_1_p_2_theta}--\eqref{eq:conv_energy_with_theta_3}, we conclude that for almost all $t\in (0,T)$
\begin{equation}\label{eq:summary_step_convergences_theta}
\begin{split}
\limsup_{\theta \to 0} \int_{0}^t \int_\Omega  S^{\theta}(\tau, x, Du^\theta):D(\psi(x)&(u^{\theta}+\nabla p_h^\theta)(\tau,x)) \diff x \diff \tau \leq\\
&\leq \int_{0}^t \int_\Omega  \chi(\tau, x):D(\psi(x)(u + \nabla p_h)(\tau,x)) \diff x \diff \tau.
\end{split}
\end{equation}

\noindent \underline{Step 5: monotonicity inequality.} In this step, we will prove for a.e. $t \in (0,T)$ and $\psi \in C_c^{\infty}(\Omega)$ we have
\begin{align}\label{conv:monotonicity_trick_0}
    \limsup_{\theta \to 0}\int_0^t\int_\Omega S(t, x, Du^\theta): Du^\theta \psi(x)\diff x\diff \tau \leq \int_0^t \int_\Omega \chi(t, x) : Du\,\psi(x)\diff x\diff \tau.
\end{align}

\noindent We decompose term on the (LHS) of \eqref{eq:summary_step_convergences_theta} into six parts $X_1$, $X_2$, $X_3$, $X_4$, $X_5$, $X_6$ as follows:
\begin{align*}
& \int_{0}^t \int_\Omega  S^{\theta}(\tau, x, Du^\theta):D(\psi(x)(u^{\theta} + \nabla p_h^\theta)(\tau,x)) \diff x \diff \tau = \\
&  = \int_{0}^t \int_\Omega S(\tau, x, Du^\theta):Du^{\theta} \psi(x)  \diff x \diff \tau + \int_{0}^t \int_\Omega S(\tau, x, Du^\theta):[D(\nabla p_h^\theta)\,\psi(x)] \diff x \diff \tau\\
& \phantom{ = } + \int_{0}^t \int_\Omega S(\tau, x, Du^\theta):[\nabla\psi(x) \otimes (u^{\theta} + \nabla p_h^\theta)] \diff x \diff \tau + \int_{0}^t \int_\Omega \theta \, \nabla_\xi m(|Du^{\theta}|):D u^{\theta} \, \psi(x) \diff x\diff \tau \\
& \phantom{ = } + \int_{0}^t \int_\Omega \theta \, \nabla_\xi m(|Du^{\theta}|):D (\nabla p_h^{\theta}) \, \psi(x) \diff x\diff \tau + \int_{0}^t \int_\Omega \theta \, \nabla_\xi m(|Du^{\theta}|):(\nabla \psi(x) \otimes (u^{\theta} + \nabla p_h^\theta))  \diff x \diff \tau\\
&=: X_1 + X_2 + X_3 + X_4 + X_5 + X_6.
\end{align*}
Term $X_1$ is the one we want to estimate. For term $X_2$, we have
\begin{equation}\label{eq:conv_theta_term_X2}
    \int_0^t \int_\Omega S(\tau, x, Du^\theta) : [D(\nabla p_h^\theta) \psi(x)]\diff x\diff \tau \to \int_0^t \int_\Omega \chi(\tau,x) : [D(\nabla p_h)\psi(x)]\diff x\diff \tau
\end{equation}
because $S(t,x,Du^{\theta}) \rightharpoonup \chi$ in $L^{s'(t,x)}$ so that $S(t,x,Du^{\theta}) \rightharpoonup \chi$ in $L^{s'_{\max}}_{t,x}$ and $D(\nabla p_h^\theta) \to D(\nabla p_h)$ in $L^{s_{\max}}_{t,x, loc}$. For $X_3$ we claim that
\begin{equation}\label{eq:conv_theta_term_X3}
\int_0^t \int_\Omega S(\tau, x, Du^\theta) : [\nabla\psi \otimes (u^\theta + \nabla p_h^\theta)]\diff x\diff \tau \to \int_0^t \int_\Omega \chi(\tau,x) : [\nabla \psi\otimes (u + \nabla p_h)]\diff x\diff \tau.
\end{equation}
To prove this we write $1 = \sum_{i=1}^N \zeta_i$ where $\{\zeta_i\}$ is the partition of unity from Notation \ref{not:zeta} so that we only need to study the integral
$$
\int_0^t \int_\Omega \zeta_i\, S(\tau, x, Du^\theta) : [\nabla\psi\otimes (u^\theta + \nabla p_h^\theta)]\diff x\diff \tau.
$$
As $\zeta_i$ is supported in $\mathcal{B}^i_r$ we can use weak convergence of $S(\tau, x, Du^\theta)$ in $L^{r_i'(t)}_{t,x}$ from \ref{conv_item_weak_monotone_operator_S} and the strong convergence $u^{\theta} + \nabla p_h^\theta \to u + \nabla p_h$ in $L^{r_i(t)}_{t,x,loc}$ from \ref{conv_item_est_add_1} and \ref{conv_item_strongu_local} (this uses also $R_i(t)-r_i(t)\geq \frac{s_{\min}}{d}$).\\

\noindent Next, for the terms $X_4$, $X_5$, $X_6$ we have
\begin{equation}\label{eq:X4_nonnegativity_and_X5}
X_4 \geq 0, \qquad \qquad X_5 \to 0, \qquad \qquad X_6 \to 0.
\end{equation}
where the convergence $X_5 \to 0$ follows immiedately from \ref{conv_item_L1_weak} and estimate \ref{lemmathet_est6}. Concerning $X_6$, the argument is the same as in \eqref{eq:convergence_term_p_4_theta} because we have exactly the same integrability of $\theta \, \nabla_\xi m(|Du^{\theta}|)$ as of $p_4^{\theta}$. Plugging \eqref{eq:conv_theta_term_X2}--\eqref{eq:X4_nonnegativity_and_X5} into \eqref{eq:summary_step_convergences_theta} we obtain \eqref{conv:monotonicity_trick_0}.

\noindent \underline{Step 6: conclusion by monotonicity trick.} By the assumption \ref{monotonicity_stress_tensor} we have
\begin{align}\label{eq:monoton_inequality_with_theta_and_eta}
    \int_{\Omega_T}(S(t, x, Du^\theta) - S(t, x, \eta)) : (Du^\theta - \eta) \, \psi(x) \geq 0
\end{align}
for any $\eta\in L^\infty_t L^\infty_x$. Now, let us study limits of two terms appearing in \eqref{eq:monoton_inequality_with_theta_and_eta}. First, we claim
\begin{align}\label{conv:monotonicity_trick_1}
    \int_{\Omega_T}S(t, x, \eta) :  Du^\theta \, \psi(x) \diff x\diff t\to \int_{\Omega_T}S(t, x, \eta) : Du \,\psi(x)\diff x\diff t.
\end{align}
Indeed, since $\eta\in L^\infty_t L^{\infty}_x$, then $S(t, x, \eta)\in L^\infty_t L^{\infty}_x$ so that \eqref{conv:monotonicity_trick_1} follows by weak convergence \ref{conv_item_Dutheta}. Second, as a direct consequence of \ref{conv_item_weak_monotone_operator_S} we have
\begin{align}\label{conv:monotonicity_trick_2}
    \int_{\Omega_T}S(t, x, Du^\theta) :  \eta \, \psi(x)\diff x\diff t \to \int_{\Omega_T}\chi(t, x) : \eta \, \psi(x)\diff x\diff t.
\end{align}
Hence, using \eqref{conv:monotonicity_trick_0}, \eqref{conv:monotonicity_trick_1} and \eqref{conv:monotonicity_trick_2}, we may take $\limsup_{\theta \to 0}$ in \eqref{eq:monoton_inequality_with_theta_and_eta} to deduce
\begin{align*}
    \int_{\Omega_T}(\chi(t, x) - S(t, x, \eta)) : (Du - \eta) \, \psi(x)\diff x\diff t \geq 0.
\end{align*}
Using Lemma \ref{res:monot_trick} (Minty's monotonicity trick), we finally obtain $\chi(t,x) = S(t,x,Du)$ a.e..
\end{proof}

\appendix

\section{Musielak-Orlicz spaces \texorpdfstring{$L^{s(t, x)}(\Omega_T)$}{Ls}}\label{app:musielaki}

\noindent Here we mark some of the basic properties of the variable exponent spaces $L^{s(t, x)}(\Omega_T)$. For the more details on the Musielak-Orlicz spaces, see \cite{chlebicka2019book}. We start with the definition

\begin{Def}\label{def:musielak_exponent_space}
Given a measurable function $s(t,x): \Omega_T \to [1,\infty)$, we let
\begin{align*}
L^{s(t,x)}(\Omega_T) = \left\{ \xi: \Omega_T \to \mathbb{R}^d: \mbox{ there is } \lambda>0 \mbox{ such that } \int_{\Omega_T} \left| \frac{\xi(t,x)}{\lambda}\right|^{s(t,x)} \diff x \diff t < \infty \right\}.
\end{align*}
Moreover if $s(t, x)$ satisfies the boundedness conditions \ref{ass:usual_bounds_exp} then this definition is equivalent to
$$
L^{s(t,x)}(\Omega_T) = \left\{ \xi: \Omega_T \to \mathbb{R}^d:  \int_{\Omega_T} \left| {\xi(t,x)}\right|^{s(t,x)} \diff x \diff t < \infty \right\}.
$$
\end{Def}

\noindent Variable exponent spaces are Banach with the following norm

\begin{thm}
Let
\begin{equation}\label{intro:norm}
\| \xi \|_{L^{s(t,x)}} = \inf \left\{\lambda>0: \int_{\Omega_T} \left| \frac{\xi(t,x)}{\lambda}\right|^{s(t,x)} \diff x \diff t \leq 1 \right\}.
\end{equation}
Then, $\left(L^{s(t, x)}, \|\cdot\|_{L^{s(t, x)}}\right)$ is a Banach space.
\end{thm}

\noindent We will be interested in the two kinds of convergences in the aformentioned spaces

\begin{Def}
We say that $\xi_n$ converges strongly to $\xi$ (denoted $\xi_n \to \xi$) in $L^{s(t, x)}(\Omega_T)$, if
$$
\|\xi_n - \xi\|_{L^{s(t, x)}}\rightarrow 0
$$
and that $\xi_n$ converges modularly to $\xi$, if there exists $\lambda > 0$ such that
$$
\int_{\Omega_T} \left| \frac{\xi_n(t,x) - \xi(t,x)}{\lambda} \right|^{s(t,x)} \diff x \diff t \to 0.
$$
\end{Def}

\noindent Modular and strong convergences are connected in the following form

\begin{thm} The following equivalence holds true:
\begin{align*}
    \|\xi_n - \xi\|_{L^{s(t, x)}} \to 0 \Longleftrightarrow \int_{\Omega_T} \left| \frac{\xi_n(t,x) - \xi(t,x)}{\lambda} \right|^{s(t,x)} \diff x \diff t \to 0 \text{ for every }\lambda >0
\end{align*}
\end{thm}

\begin{cor}
If $s(t, x)$ satisfies the boundedness conditions \ref{ass:usual_bounds_exp}, then strong and modular convergences are equivalent.
\end{cor}

\noindent The Theorem below makes a connection between modular convergence and the convergences of the products as formulated below.

\begin{thm}\label{thm:modular_l1}\textbf{\textup{(Proposition 2.2, \cite{gwiazda2008onnonnewtonian})}}
Assume that the function $s(t, x)$ satisfies Assumption \ref{ass:exponent_cont_space}. Presuppose also that $\phi_n \to \phi$ modularly in $L^{s(t, x)}(\Omega_T)$ and $\psi_n\to \psi$ modularly in $L^{s'(t, x)}(\Omega_T)$. Then, $\phi_n\,\psi_n \to \phi\,\psi$ in $L^1(\Omega_T)$.
\end{thm}

\section{Poisson equation in \texorpdfstring{$L^p(\R^d)$}{Lp}}

\noindent The classical theory (see \cite[Theorem 1]{evans1998partial}, p. 23) states, that given $f \in C_c^{\infty}(\R^d)$ the equation
$$
-\Delta u = f \mbox{ on } \R^d,\qquad \qquad u(x) \to 0 \mbox{ as } |x|\to\infty.
$$
admits the unique smooth solution given via Newtonian potential
$$
u(x) = \Gamma \ast f, \qquad \qquad \Gamma(x) = \begin{cases}
-\frac{1}{2\pi}\,\log{|x|} & \mbox{ if } d=2,\\
\frac{1}{d(d-2)\alpha(d)}|x|^{2-d} & \mbox{ if } d>3,
\end{cases}
$$
where $\alpha(d)$ is the volume of the unit ball. Here we focus on the theory for $L^p(\R^d)$ spaces:
\begin{thm}\label{thm:poissonLp}
Let $g \in L^p(\Omega)$ and consider its extension to $\R^d$ with 0. Then, there exists the unique distributional solution to
$$
-\Delta u = \DIV \DIV g \mbox{ in } \R^d, \qquad g \in L^p(\R^d).
$$
Moreover, $\|u\|_{L^p(\R^d)} \leq C\, \|g\|_{L^p(\Omega)}$.
\end{thm}

\noindent To prove the theorem, we will need a few simple lemmas.

\begin{lem}[decay estimates for the Poisson's equation]
Let $f \in C_c^{\infty}(\R^d)$ and let $R$ be such that $\mbox{supp} f \subset B_R$.
\begin{itemize}
    \item[(A)] Let $u = \Gamma \ast \DIV f$. Then, there is a constant depending on $\|f\|_{L^1}$ such that for $|x| > 2R$
    $$
    |u(x)| \leq C\, |x|^{1-d}, \qquad \qquad |\nabla u| \leq C \, |x|^{-d}.
    $$
    \item[(B)] Let $u = \Gamma \ast f$. Then, there is a constant depending on $\|f\|_{L^1}$ such that for $|x| > 2R$
    $$
    |u(x)| \leq \begin{cases} C\, |x|^{2-d} &\mbox{ if } d> 2\\
    C\, \log||x| - R| &\mbox{ if } d = 2.
    \end{cases}
    \qquad \qquad |\nabla u| \leq C \, |x|^{1-d}.
    $$
\end{itemize}
\end{lem}
\begin{proof}
First, we consider (A). Let $|x| > 2R$. We observe that
$$
\Gamma \ast \DIV f(x) = \int_{B_R} \Gamma(x-y) \DIV f(y) \diff y =
- \int_{B_R} \nabla \Gamma(x-y)\, f(y) \diff y,
$$
where integration by parts is justified because we are away from singularity of $\Gamma$ and $f$ is compactly supported. Now,
$$
|x-y| \geq |x| - |y| \geq |x| - R \geq |x|/2.
$$
The conclusion follows because $\nabla \Gamma(x-y)$ is of the form $\frac{C}{|x-y|^{d-1}}$. The estimate on $\nabla u$ is proved in exactly the same way: this time we note that second order derivatives of $\Gamma(x-y)$ are of the form $\frac{C}{|x-y|^d}$. Finally, the proof of (B) is completely analogous: the only difference is that we cannot pass the divergence operator from $f$ to $\Gamma$ which results in a worse decay estimates.
\end{proof}

\begin{lem}[$L^p$ global estimate]\label{lem:globalLp_poisson}
Let $g \in C_c^{\infty}(\R^d)$ and let $u = \Gamma \ast \DIV \DIV g $. Then, $\|u\|_{L^p(\R^d)} \leq C\, \|g\|_{L^p(\R^d)}$.
\end{lem}
\begin{proof}
Let $\varphi \in C_c^{\infty}(\R^d)$ with $\|\varphi\|_{L^{p'}(\R^d)} \leq 1$. Let $\phi_{\varphi}:= \Gamma \ast \varphi$. Then, $\| D^2 \phi_{\varphi}\|_{L^{p'}} \leq C$, cf. \cite[Theorem 3.5]{MR1616087}. We have
$$
\int_{\R^d} u(x) \, \varphi(x) \diff x = - \int_{\R^d} u(x) \, \Delta \phi_{\varphi}(x) \diff x =
\int_{\R^d} \nabla u(x) \, \nabla \phi_{\varphi}(x) \diff x.
$$
The integration by parts is justified here as for large $R$:
$$
\left|\int_{\partial B_R} u \, \nabla \phi_{\varphi} \cdot {\bf{n}} \diff S\right| \leq C\, R^{1-d} \, R^{1-d} \, R^{d-1} \to 0.
$$
and the boundary term disappears. Furthermore,
$$
\int_{\R^d} \nabla u(x) \, \nabla \phi_{\varphi}(x) \diff x =
-\int_{\R^d} \Delta u(x) \,  \phi_{\varphi}(x) \diff x = \int_{\R^d} \DIV \DIV g(x)\,  \phi_{\varphi}(x) \diff x.
$$
Again, to justify the integration by parts we just compute
$$
\left|\int_{\partial B_R} \phi_{\varphi} \, \nabla u  \cdot {\bf{n}} \diff S\right| \leq \begin{cases}
\log(R) \, R^{-2} \, R & (\mbox{ if } d = 2) \\
R^{2-d} \, R^{-d} \, R^{d-1} & (\mbox{ if } d > 2)
\end{cases} \qquad \to 0.
$$
Finally, by compact support of $g$, we can integrate by parts twice to deduce
$$
\left|\int_{\R^d} \DIV \DIV g(x)\,  \phi_{\varphi}(x) \diff x\right| \leq
C \, \|g\|_{L^p(\R^d)}
$$
uniformly in $\varphi$. The conclusion follows.
\end{proof}

\noindent We are in position to prove Theorem \ref{thm:poissonLp}.
\begin{proof}[Proof of Theorem \ref{thm:poissonLp}]

\noindent To see existence, we consider usual mollification $g_{\varepsilon}$ of $g$ and define $u_{\varepsilon} = \Gamma \ast g_{\varepsilon}$. Then, Lemma \ref{lem:globalLp_poisson} gives sufficient bounds to pass to the limit in the distributional formulation. The needed estimate follows from the weak lower-semicontinuity of the norm.\\

\noindent For the uniqueness part assume, that there are two solutions $u_1, u_2 \in L^p(\R^d)$. Then from Weyl's lemma (cf. \cite{simader1996thedirichlet}) $u:= u_1 - u_2$ is a harmonic function and so, it is smooth. Then, the mean value property implies
$$
|u(x)| \leq \frac{1}{\alpha(d) R^d} \int_{B_R(x)} |u(y)| \diff y \leq \frac{\|u\|_{L^p(\R^d)}}{\alpha(d)^p} \, R^{-d/p}.
$$
Sending $R\to \infty$ we deduce $u = 0$.

\end{proof}

\section{Useful results}

\begin{lem}\label{thm:interpolation}
Suppose that $v \in L^{\infty}_t L^2_x\cap L^q_t W^{1,q}_x$ and $q \geq 2$. Then, $v \in L^{r_0}_t L^{r_0}_x$ where $r_0 = q \left(1 + \frac{2}{d} \right)$ and
$$
\|v\|_{L^{r_0}_t L^{r_0}_x} \leq C(\|v\|_{L^{\infty}_t L^2_x}, \|v\|_{L^q_t W^{1,q}_x}).
$$
\end{lem}

\begin{lem}\label{L2convlemma}
Let $1 \leq p < \infty$ and $\{u_n\}_{n \in \N}$ be a sequence such that $u_n \to u$ in $L^p_t L^p_x$. Then, there exists a subsequence $\{u_{n_k}\}_{k \in \N}$, such that
$$
\mbox{for a.e. } t \in (0,T) \qquad u_{n_k}(t,x) \to u(t,x) \mbox{ in } L^p_x
$$
Moreover, if $\{u_{n_k}\}_{k \in \N}$ is bounded in $L^{\infty}_t L^2_x$ we have for a.e. $t \in (0,T)$
$$
\int_{\Omega} |u(t,x)|^2 \diff x \leq \liminf_{k \to \infty} \int_{\Omega} |u_{n_k}(t,x)|^2 \diff x
$$
\end{lem}

\begin{lem}\label{lem:double_moll_conv}
Let $v \in L^1_t L^1_x$ and $\eta_{\varepsilon}$ be as in Definition \ref{res:mol_in_sp}. Then
$$
\int_{B_{\varepsilon}(0)} \int_{B_{\varepsilon}(0)} v(t,x-y-z) \, \eta
_{\varepsilon}(y) \, \eta_{\varepsilon}(z) \diff y \diff z \to v(t,x) \mbox{ in } L^1_t L^1_x.
$$
\end{lem}

\begin{lem}\label{lem:conv_mollifier_with_weight}
Let $f \in L^2_x$ and $\psi \in L^{\infty}(\Omega)$. Then
$$
\int_{\Omega} |f^{\varepsilon}(x)|^2 \, \psi(x) \diff x \to \int_{\Omega} |f(x)|^2 \, \psi(x) \diff x.
$$
\end{lem}

\begin{lem}\textup{\textbf{(Ne\v{c}as theorem about negative norms \cite[Lemma 2.2.2]{MR3013225})}}\label{lem:deRham}
Let $\Omega\subset \mathbb{R}^d$, $d\geq 2$, be a bounded Lipchitz domain,  let $1 < q < \infty$. Suppose $f \in W^{-1, q}(\Omega)^d$ satisfies
\begin{align*}
    f(v) = 0 \text{ for all } v\in C^\infty_{0}(\Omega) \text{, }\DIV v = 0
\end{align*}
Then there exists a unique $p\in L^q(\Omega)$ satisfying
$$
\int_{\Omega} p \diff x = 0, \qquad f = \nabla p
$$
in the sense of distributions. Moreover,
\begin{align}\label{bound_on_p_deRham}
    \|p\|_{L^q_x} \leq C\|f\|_{W^{-1, q}(\Omega)^d}
\end{align}
with some constant $C = C(q, \Omega_0, \Omega) > 0$.
\end{lem}
\noindent We remark that $W^{-1,q}(\Omega)$ is the dual space of $W^{1,q'}(\Omega)$.

\begin{lem}\label{res:monot_trick}\textup{\textbf{(Lemma 2.16, \cite{bulicek2021parabolic})}}
Let $S$ be as in Assumption \ref{ass:stress_tensor}. Assume there are $\chi \in L^{s'(t,x)}(\Omega_T)$ and $\xi \in L^{s(t,x)}(\Omega_T)$, such that
\begin{equation*}
\int_{\Omega_T} \left(\chi - S(t,x,\eta)\right) : (\xi - \eta)\, \psi(x)\diff t \diff x \geq 0
\end{equation*}
for all $\eta \in L^{\infty}(\Omega_T; \R^d)$ and $\psi \in C_0^{\infty}(\Omega)$ with $0 \leq \psi \leq 1$. Then,
$$
S(t,x,\xi) = \chi(t,x) \mbox{ a.e. in } \Omega_T.
$$
\end{lem}

\begin{lem}\label{aubin-lions}\textup{\textbf{(Generalized Aubin--Lions lemma, \cite[Lemma 7.7]{MR3014456})}}
Denote by
$$
W^{1, p, q}(I; X_1, X_2) := \left\{u\in L^p(I; X_1); \frac{du}{dt}\in L^q(I; X_2)\right\}
$$
Then if $X_1$ is a separable, reflexive Banach space, $X_2$ is a Banach space and $X_3$ is a metrizable locally convex Hausdorff space, $X_1$ embeds compactly into $X_2$, $X_2$ embeds continuously into $X_3$, $1 < p <\infty$ and $1\leq q\leq \infty$, we have
$$
W^{1, p, q}(I; X_1, X_3) \text{ embeds compactly into }L^p(I; X_2)
$$
In particular any bounded sequence in $W^{1, p, q}(I; X_1, X_3)$ has a convergent subsequence in $L^p(I; X_2)$.
\end{lem}

\bibliographystyle{abbrv}
\bibliography{parpde_mo_discmodul}
\end{document}